\documentclass[12pt]{amsart}
\usepackage{amsmath,amssymb,amsbsy,amsfonts,latexsym,amsopn,amstext,
                                               amsxtra,euscript,amscd,bm}
\usepackage[margin=1in]{geometry}                    
\usepackage{url}
\usepackage[colorlinks,linkcolor=blue,anchorcolor=blue,citecolor=blue]{hyperref}
\usepackage{color}
\usepackage{enumerate}

\begin{document}

\newtheorem{theorem}{Theorem}[section]
\newtheorem{lemma}[theorem]{Lemma}
\newtheorem{conjecture}[theorem]{Conjecture}
\newtheorem{proposition}[theorem]{Proposition}
\newtheorem{corollary}[theorem]{Corollary}
\newtheorem{claim}[theorem]{Claim}
\newtheorem{example}[theorem]{Example}
\theoremstyle{definition}
\newtheorem{remark}[theorem]{Remark}
\newtheorem{definition}[theorem]{Definition}

\def\F{{\mathbb F}}
\def\R{{\mathbb R}}
\def\Z{{\mathbb Z}}
\def\C{{\mathbb C}}
\def\N{{\mathbb N}}
\def\O{{\mathbb O}}
\def\B{{\mathcal B}}
\def\Hy{{\mathcal H}}
\def\S{{\mathcal S}}

\newcommand{\rst}[1]{\ensuremath{{\mathbin\upharpoonright}\raise-.5ex\hbox{$#1$}}}


\title{Almost orthogonal subsets of vector spaces over finite fields}

 \author[A. Mohammadi] {Ali Mohammadi}

\address{A.M.: School of Mathematics, Institute for Research in Fundamental Sciences (IPM),
Tehran, Iran}
\email{a.mohammadi@ipm.ir}

 \author[G.\ Petridis]{Giorgis Petridis}
\address{G.P.: Department of Mathematics,
University of Georgia,  Athens, GA, 30602 USA}
\email{giorgis@cantab.net}

\pagenumbering{arabic}

\begin{abstract}
We prove various results on the size and structure of subsets of vector spaces over finite fields which, in some sense, have too many mutually orthogonal pairs of vectors. In particular, we obtain sharp finite field variants of a theorem of Rosenfeld and an almost version of a theorem of Berlekamp.
\end{abstract}

\maketitle

\section{Introduction}
\subsection{Background}
An \emph{orthogonal set} of vectors in $\R^n$, that is a set of non-zero vectors with the property that every pair of distinct vectors is mutually orthogonal, is linearly independent and therefore contains at most $n$ elements. Erd\H{o}s asked the question of determining the maximum size of a set of \emph{almost orthogonal} vectors: a set of non-zero vectors with the property that among any three distinct vectors, at least two are mutually orthogonal \cite{NeRo}. The union of two disjoint orthogonal sets is an almost orthogonal set of size $2n$. Rosenfeld, confirming a belief of Erd\H{o}s, proved that $2n$ is the maximum size of an almost orthogonal subset of $\R^n$~\cite{Rose}. Deaett gave a short and elegant proof of Rosenfeld's theorem~\cite{Dea}, which has similarities with an argument of Pudl\'ak~\cite{Pud}. Deaett also proved that for dimension 4 and lower every almost orthogonal set of maximum size is the union of two orthogonal sets; and provided examples in dimension 5 and higher of almost orthogonal sets of maximum size that are not the union of two disjoint orthogonal sets.

In $\C^n$, the existence of self-orthogonal vectors (like $(1,i) \in \C^2)$ changes the answer to both questions. Even in dimension two, the span of $(1,i)$ is an uncountable set of orthogonal vectors. Deaett proved, however, that in $\C^n$ equipped with the Hermitian inner product, the maximum number of almost orthogonal vectors is $2n$, generalising Rosenfeld's theorem~\cite{Dea}.

Both questions have also been investigated over finite fields. The size of the largest orthogonal set in $(\Z/(2\Z))^n$ was determined by Berlekamp~\cite{Ber} and the size of the largest orthogonal set in $(\Z/(p\Z))^n$ for primes $p$ was determined by Zame~\cite{Zame}. There are similarities in their methods. The question Berlekamp answered is equivalent to solving another question of Erd\H{o}s: determining the size of the largest family of subsets of $\{1, 2, \dots, n\}$ with the property that every two distinct elements have even intersection. Erd\H{o}s' question was solved independently by Graver~\cite{Ber,Grav}.

A key to Berlekamp's and Zame's arguments is the existence of self-orthogonal vectors. Self-orthogonal vectors exist over any finite field when the dimension is at least 3 or, in dimension 2, when the order of the field is 2 or is congruent to 1 modulo 4. There are further intricacies when working in vector spaces over finite fields. For example, as is detailed in the next subsection, in dimension 6, the dot product is equivalent to the symmetric bilinear form $(\bm{x} , \bm{y}) \mapsto x_1 y_1 - x_2 y_2 + x_3 y_3 - x_4 y_4 + x_5 y_5 - x_6 y_6$ when the order of the field is congruent to 1 modulo 4, and to the symmetric bilinear form $(\bm{x} , \bm{y}) \mapsto x_1 y_1 - x_2 y_2 + x_3 y_3 - x_4 y_4 + x_5 y_5 + x_6 y_6$ when the order of the field is congruent to 3 modulo 4. The difference points to the fact that the largest orthogonal subspace has dimension that depends on the order of the field~\cite{Vinh}. In addition to this, expressing the dot product in these equivalent ways has the advantage that it makes clear the existence of self-orthogonal vectors. 

There does not seem to be a significant difference between studying the dot product and studying any symmetric non-degenerate bilinear form and this is the approach taken in the literature recently. Ahmadi and Mohammadian~\cite{AhmMoh}, using an argument similar to Berlekamp, determined the size of the largest orthogonal set with respect to any non-degenerate symmetric bilinear form over fields of odd order (see also~\cite{HIKR, Vinh}). Ahmadi and Mohammadian also made progress on the question of determining the size of the largest almost orthogonal set with respect to any non-degenerate symmetric bilinear form.

The main purpose of this paper is to determine the size of the largest almost orthogonal set with respect to any bilinear form in any vector space over any sufficiently large finite field of odd order; and also in $(\Z/(2\Z))^n$ for sufficiently large $n$. It is worth recording here that, unlike Rosenfeld's and Deaett's theorems, it is not always the case that the size of the largest almost orthogonal set equals twice the size of the largest orthogonal set. The $(\Z/(2\Z))^n$ question has a set system formulation that can be thought of as an ``almost'' version of Berlekamp's theorem: determine the size of the largest family of subsets of $\{1, 2, \dots, n\}$ with the property that among every three distinct elements, at least two have even intersection. We show that the size of the largest family almost doubles. We also investigate the finite field analogue of another question that Erd\H{o}s asked for Euclidean space: determine the maximum size of subsets of $\R^n$ with the property that among any $k$ of their elements, at least two are mutually orthogonal~\cite{AlSz,FuSt}.

\subsection{Notation and definitions}
Throughout the paper, we use $m$ and $n$ to be positive integers, $p$ a prime and $q = p^m$. We also use $\F_q$ to denote a finite field of order $q$ and  write $\F_q^* = \F_q\setminus\{0\}$. A bilinear form over $\F_q^n$ is a mapping $\mathcal{B}:\F_q^n\times \F_q^n \rightarrow \F_q$, which takes the form
\[
\B(\bm{x}, \bm{y}) = \bm{x}^{T}A\bm{y}, \ \text{for all }\bm{x}, \bm{y}\in \F_q^n,
\]
for some $n\times n$ matrix $A$ over $\F_q$. We say $\B$ is symmetric if $A$ is a symmetric matrix and say $\B$ is degenerate if $\text{det}(A) = 0$. We call two bilinear forms equivalent if their corresponding matrices are equivalent ($A, B$ are equivalent if $A = M^T B M$ for an invertible matrix $M$). 

Fix a non-square element $\gamma \in \F_q$ and let $k= \lfloor\frac{n}{2}\rfloor$. For any bilinear form $\mathcal{B}$ over $\F_q^n$, with associated matrix $A$, we define
\[
\varepsilon(\mathcal{B}) = \begin{cases} 0, &\mbox{if } \det(A)=0; \\
1, & \mbox{if } k \ \text{is even and } \det(A)\ \text{is a non-zero square, or}\\
& \mbox{if } k \ \text{is odd  and } -\det(A) \ \text{is a non-zero square};\\
\gamma, & \mbox{if } k \ \text{is even and } \det(A) \ \text{is a non-square, or}\\
& \mbox{if } k \ \text{is odd and } -\det(A) \ \text{is a non-square.}\end{cases}
\]

For odd $q$, by a result in \cite[p. 79]{Grove}, which also appears as \cite[Theorem~1]{AhmMoh}, any non-degenerate symmetric bilinear form, $\B$, over $\F_q^n$ is equivalent to the form
\begin{equation}
\label{eqn:1stBF}
    (\bm{x}, \bm{y}) \mapsto x_1y_1 - x_2y_2+ \dots + x_{n-2}y_{n-2}-x_{n-1}y_{n-1}+\varepsilon(\B)x_n y_n
\end{equation}
 for odd $n$ and is equivalent to the form
\begin{equation}
\label{eqn:2ndBF}
    (\bm{x}, \bm{y}) \mapsto x_1y_1 - x_2y_2+ \dots + x_{n-3}y_{n-3}-x_{n-2}y_{n-2}+x_{n-1}y_{n-1}-\varepsilon(\B)x_n y_n
\end{equation}
for even $n$, where $\bm{x} = (x_1, \dots, x_n)$ and $\bm{y} = (y_1, \dots, y_n)$.

Given finite-dimensional vector spaces $V_1$ and $V_2$ over $\F_q$, with $n_i = \text{dim}(V_i)$ for $i=1,2$, we define the \emph{direct sum} $V_1\oplus V_2$ to be the vector space $V_1\times V_2$, which may be identified by $\F_q^{n_1+n_2}$. Furthermore, if $M_1$ and $M_2$ are matrices corresponding to bilinear forms over $\F_q^{n_1}$ and $\F_q^{n_2}$ respectively, we define the matrix $M_1\oplus M_2$ by
\[
M_1\oplus M_2 = \begin{pmatrix} M_1 & 0 \\ 0 & M_2\end{pmatrix},
\]
which gives rise to a bilinear form over $\F_q^{n_1+n_2}.$

For $q=2$, one can infer from (5) in \cite[p. 7]{Kne} that every non-degenerate symmetric bilinear form in odd dimension is equivalent to the dot product that arises from the $n \times n$ identity matrix $I_n$. In even dimensions every non-degenerate symmetric bilinear form is either equivalent to the dot product or to the \emph{hyperbolic form}\label{page:hyperbolic} $\Hy$ that arises from the matrix $H \oplus \dots \oplus H$, where
\[
H = \begin{pmatrix} 0 & 1 \\ 1 & 0 \end{pmatrix}.
\]

\begin{definition}
\label{def:OrthogonalSets}
We refer to two vectors $\bm{v_1}, \bm{v_2} \in \F_q^n\setminus\{\bm{0}\}$ as mutually \emph{orthogonal} if $\mathcal{B}(\bm{v_1}, \bm{v_2}) = 0$. If $\bm{v}\in \F_q^n\setminus\{\bm{0}\}$, satisfies $\mathcal{B}(\bm{v}, \bm{v}) = 0$, we refer to it as \emph{self-orthogonal}.  We call a subset $S\subset \F_q^n\setminus\{\bm{0}\}$ an \emph{orthogonal set} if every distinct pair of elements of $S$ are mutually orthogonal and we say $S\subset \F_q^n\setminus\{\bm{0}\}$ is \emph{$(k, l)$-orthogonal} if for any $k$ vectors in $S$ at least $l$ of them are pairwise mutually orthogonal. Furthermore, we call a subspace $V\subset \F_q^n$ an \emph{orthogonal subspace} if $V\setminus\{\bm{0}\}$ is an orthogonal set. We denote by $\S_{k,l} = \S_{k,l}(q, n, \B)$ the maximum size of any $(k, l)$-orthogonal subset of $\F_q^n$ with respect to $\B$. 
\end{definition}

Given a set $X\subset \F_q^n$, we use $\langle X\rangle$ to denote the subspace of $\F_q^n$ generated by $X$ and write $\langle \bm{v}_1, \cdots, \bm{v}_k\rangle$, instead of $\langle \{\bm{v}_1, \cdots, \bm{v}_k\}\rangle$. We also define the \emph{orthogonal complement} of $X$ by $X^{\perp} = \{\bm{v}\in \F_q^n: \B(\bm{v}, \bm{x}) = 0\ \text{for all } \bm{x}\in X\}$, which constitutes a subspace of $\F_q^n$.

Finally, for sets $S, S_1, \dots, S_k$, where $k\geq 2$, we write $S=S_1\sqcup S_2 \sqcup \cdots \sqcup S_k$ to mean firstly that $S = S_1\cup S_2 \cup \cdots \cup S_k$ and secondly that $S_i\cap S_j = \emptyset$ for $1\leq i<j\leq k$.
\subsection{Previous results for almost orthogonal sets}

In \cite[Examples~12--15]{AhmMoh}, explicit examples of $(3, 2)$-orthogonal sets are provided for odd $q$, showing that
\begin{equation}
\label{eqn:AhmMohLB}
\S_{3,2}(q,n, \B)\geq 
\begin{cases} 2q^{\frac{n-1}{2}}, &\mbox{if } n \ \text{is odd and }\varepsilon(\B) = 1;\\
2q^{\frac{n-1}{2}}-q+1, &\mbox{if } n \ \text{is odd and }\varepsilon(\B) = \gamma;\\
2q^{\frac{n}{2}}-q-1, &\mbox{if } n \ \text{is even and}\ \varepsilon(\mathcal{B}) = 1; \\
 2q^{\frac{n}{2}-1}+ 2, &\mbox{if } n \ \text{is even and}\ \varepsilon(\mathcal{B}) = \gamma .\end{cases}
\end{equation}
The examples also work for $q=2$, showing that
\begin{equation}
\label{eqn:AhmMohLB2}
\S_{3,2}(2,n, \cdot)\geq 
\begin{cases} 2^{\frac{n+1}{2}}, &\mbox{if $n$ is odd};\\
2^{\frac{n}{2}+1}-3, &\mbox{if  $n$ is even} .\end{cases}
\end{equation}

The authors further conjectured, in \cite[Conjectures~11 and 16]{AhmMoh}, that the inequalities in \eqref{eqn:AhmMohLB} and \eqref{eqn:AhmMohLB2} could be replaced by equalities and outlined a proof, in \cite[Theorem~17]{AhmMoh}, that for all $n$ and either choices \eqref{eqn:1stBF} and \eqref{eqn:2ndBF} of bilinear forms, $\S_{3,2}\leq 3q^{\lfloor \frac{n}{2}\rfloor}$. 
\subsection{Improved lower bounds on $\mathcal{S}_{3,2}$} 
\label{sec:examples}
We proceed to present examples of $(3,2)$-orthogonal subsets of $\F_q^n$ for odd $q$ that have slightly more elements than the examples given in \eqref{eqn:AhmMohLB}. Theorem~\ref{thm:32OS odd}, which is stated in the next section, shows that the examples described below are of maximum size. We denote by $\{\bm{e}_1, \dots, \bm{e}_n\}$ the standard basis of $\F_q^n$.

\begin{example}
\label{ex:1}
Let $q$ be odd and $n=2k+1\geq 3$ and $\B$ satisfy $\varepsilon(\B) \in \{1,\gamma\}$. Consider the two mutually disjoint orthogonal sets
\[
S_1 = \big(\{(x_1, x_1, \dots, x_k, x_k, 0): x_1, \dots, x_k\in \F_q\}\setminus\{\bm{0}\}\big) \sqcup \{\bm{e}_{n-2}+\bm{e}_{n-1}+ \bm{e}_{n}\}
\]
and
\[
S_2 = \big(\{(x_1, -x_1, \dots, x_k, -x_k, 0): x_1, \dots, x_k\in \F_q\}\setminus\{\bm{0}\}\big) \sqcup \{\bm{e}_{n-2}-\bm{e}_{n-1}-2\varepsilon(\B)^{-1} \bm{e}_{n}\}.
\]
Then, the set $S= S_1\sqcup S_2 \sqcup \{\bm{e}_{n}\}$ is a $(3, 2)$-orthogonal set of size $2q^{k}+1= 2q^{(n-1)/2}+1$.
\end{example}

\begin{example}
\label{ex:2}
Let $q$ be odd and $n=2k\geq 2$ and $\B$ satisfy $\varepsilon(\B) =1$. Consider the two mutually disjoint orthogonal sets
\[
S_1 = \{(x_1, x_1, \dots, x_k, x_k): x_1, \dots, x_k\in \F_q\}\setminus\{\bm{0}\}
\]
and
\[
S_2 = \{(x_1, -x_1, \dots, x_k, -x_k): x_1, \dots, x_k\in \F_q\} \setminus\{\bm{0}\}.
\]
Then, the set $S= S_1\sqcup S_2$ is a $(3, 2)$-orthogonal set of size $2q^{k}-2 = 2q^{n/2}-2$.
\end{example}

\begin{example}
\label{ex:3}
Let $q$ be congruent to 3 modulo 4 and $n=2k \geq 4$ and $\B$ satisfy $\varepsilon(\B) =-1$ ($-1$ is not a square). Consider the three pairwise disjoint orthogonal sets
\begin{align*}
S_1 = \big(\{(&x_1, x_1, \dots, x_{k-1}, x_{k-1}, 0, 0): x_1, \dots, x_{k-1}\in \F_q\}\setminus\{\bm{0}\}\big) \sqcup\{\bm{e}_{n-3} + \bm{e}_{n-2}+ \bm{e}_{n-1} + \bm{e}_{n},\\  & \bm{e}_{n-3} +\bm{e}_{n-2} + \bm{e}_{n-1} - \bm{e}_{n}\}
\end{align*}
and
\begin{align*}
S_2 = \big(\{(&x_1, -x_1, \dots, x_{k-1}, -x_{k-1}, 0, 0): x_1, \dots, x_{k-1}\in \F_q\} \setminus\{\bm{0}\}\big) \sqcup\{\bm{e}_{n-3}- \bm{e}_{n-2} - \bm{e}_{n-1} - \bm{e}_n,\\ &\bm{e}_{n-3}- \bm{e}_{n-2} -\bm{e}_{n-1} + \bm{e}_n\},
\end{align*}
and
\[
S_3 = \{\bm{e}_{n-1} + \bm{e}_{n} , \bm{e}_{n-1} - \bm{e}_{n}\}.
\]
Then, the set $S= S_1\sqcup S_2\sqcup S_3$ is a $(3, 2)$-orthogonal set of size $2q^{k-1}+4=2q^{n/2-1}+4$. For $n=2$, $\{\bm{e}_1, 2 \bm{e}_1, \bm{e}_2, 2 \bm{e}_2\}$ is a $(3,2)$-orthogonal set of size 4.
\end{example}

Next, we provide examples for $q=2$. Theorem~\ref{thm:q2even} shows they are of maximum size. The following example is obtained by adding a single element to the example given by \cite[Example~12]{AhmMoh}.

\begin{example}
\label{ex:4}
Let $n=2k+1\geq 3$ and $\B$ denote the dot product in $\F_2^n$. Consider the disjoint orthogonal sets
\[
S_1 = \big(\{(0, x_1, x_1 \dots, x_k, x_k): x_1, \dots, x_k\in \F_2\} \setminus\{\bm{0}\}\big) \sqcup  \{\bm{e}_1\}
\]
and
\[
S_2 = \big(\{(x_1, x_1, \dots, x_k, x_k, 0): x_1, \dots, x_k\in \F_2\}\setminus\{\bm{0}\}\big) \sqcup \{\bm{e}_n\}.
\]
Then, the set $S= S_1\sqcup S_2  \sqcup \{\bm{e}_1 + \bm{e}_2+ \dots + \bm{e}_n\}$ is a $(3, 2)$-orthogonal set of size $2^{k+1}+1 = 2^{(n+1)/2}+1$.
\end{example}

The example below is the same as \cite[Example~14]{AhmMoh}.

\begin{example}
\label{ex:5}
Let $n=2k\geq 2$ and $\B$ denote the dot product in $\F_2^n$. Consider the disjoint orthogonal sets
\[
S_1 = \{(x_1, x_1, \dots, x_k, x_k): x_1, \dots, x_k\in \F_2\}\setminus\{\bm{0}\}
\]
and
\[
S_2 = \{(x_k, x_1, x_1, \dots, x_{k-1}, x_{k-1}, x_k): x_1, \dots, x_k\in \F_2\} \setminus\{\bm{0}\}.
\]
Then, noting $(1, 1, \dots, 1, 1) \in S_1\cap S_2$, it follows that the set $S= S_1\cup S_2$ is a $(3, 2)$-orthogonal set of size $2^{k+1}-3 = 2^{n/2+1}-3$.
\end{example}

Our final example concerns the hyperbolic form.

\begin{example}
\label{ex:6}
Let $n=2k\geq 2$ and $\Hy$ denote the hyperbolic form in $\F_2^n$. Consider the disjoint orthogonal sets
\[
S_1 = \langle \bm{e}_1, \bm{e}_3, \bm{e}_5, \dots, \bm{e}_{2k-1}\rangle \setminus\{\bm{0}\}
\]
and
\[
S_2 =  \langle \bm{e}_2, \bm{e}_4, \bm{e}_6, \dots, \bm{e}_{2k}\rangle \setminus\{\bm{0}\}.
\]
It follows that the set $S= S_1\cup S_2$ is a $(3, 2)$-orthogonal set of size $2^{k+1}-2 = 2^{n/2+1}-2$.
\end{example}

\section{Main results}

Our first main result is an upper bound on the size of $(3,2)$-orthogonal sets for odd $q$.

\begin{theorem}
\label{thm:32OS odd}
Let $n\geq 0$ be an integer and $q\geq 7$ be an odd prime power. If $S\subset\F_q^n$ is $(3, 2)$-orthogonal with respect to a non-degenerate symmetric bilinear form $\B$, then
\[
|S| \leq \begin{cases} 2q^{\frac{n-1}{2}}+1, &\mbox{if } n \ \text{is odd};\\
2q^{\frac{n}{2}}-2, &\mbox{if } n \ \text{is even and}\ \varepsilon(\mathcal{B}) = 1; \\
 2q^{\frac{n}{2}-1}+ 4, &\mbox{if } n\geq 4 \ \text{is even and}\ \varepsilon(\mathcal{B}) = \gamma ; \\
  4, &\mbox{if } n=2 \ \text{ and}\ \varepsilon(\mathcal{B}) = \gamma .\end{cases}
\]
\end{theorem}

Combined with Examples \ref{ex:1}, \ref{ex:2}, and~\ref{ex:3}, Theorem~\ref{thm:32OS odd} establishes the value of $\mathcal{S}_{3,2}$ for all $n$, sufficiently large $q$, and $\B$:
\[
\mathcal{S}_{3,2}(q,n,\B) = \begin{cases} 2q^{\frac{n-1}{2}}+1, &\mbox{if } n\geq 3 \ \text{is odd};\\
2q^{\frac{n}{2}}-2, &\mbox{if } n\geq 2 \ \text{is even and}\ \varepsilon(\mathcal{B}) = 1; \\
 2q^{\frac{n}{2}-1}+ 4, &\mbox{if } n\geq 4 \ \text{is even and}\ \varepsilon(\mathcal{B}) = \gamma ; \\
  4, &\mbox{if } n=2 \ \text{ and}\ \varepsilon(\mathcal{B}) = \gamma .\end{cases}
\]

In contrast to the Euclidean space $\R^n$, $\mathcal{S}_{3,2}$ is sometimes larger than twice the size of the largest orthogonal set (specifically when $n\geq 3$ is odd or when $n\geq 4$ is even and $\varepsilon(\B) = \gamma$). See \cite{AhmMoh} or Lemma~\ref{lem:OSUB} below for the size of the largest orthogonal set for the various possibilities of $n$, $q$, and $\B$.

The proof of Theorem~\ref{thm:32OS odd} relies on the framework developed by Ahmadi and Mohammadian~\cite{AhmMoh}. It also has similarities with the work of Berlekamp~\cite{Ber} and the paper of Deaett~\cite{Dea}. In Section~\ref{sec: character} we present a different argument for even $n$ with $\varepsilon(\B)=1$ that is based on character sum estimates. The proof in Section~\ref{sec: character} works for all odd $q$.

We also answer the corresponding question for $q=2$.

\begin{theorem}
\label{thm:q2even}
Let $n$ be an integer. If $S \subset \F_2^n$ is a $(3,2)$-orthogonal with respect to the dot product, then 
\[
    |S| \leq  \begin{cases}
    2^{\frac{n+1}{2}}+1, \quad &\text{if}\quad n\geq 21 \ \text{is odd};\\
    2^{\frac{n}{2} +1} -3, \quad &\text{if}\quad n\geq 18 \ \text{is even}.
    \end{cases}
\]
For even $n \geq 2$ and $S \subset \F_2^n$ be $(3,2)$-orthogonal with respect to the hyperbolic form $\Hy$,
\[
|S| \leq 2^{\frac{n}{2}+1} - 2.
\]
\end{theorem}

Combined with Examples \ref{ex:4}, \ref{ex:5}, and~\ref{ex:6}, Theorem~\ref{thm:q2even} establishes the value of $\mathcal{S}_{3,2}$ for all sufficiently large $n$:
\[
\mathcal{S}_{3,2}(2,n,\B) = \begin{cases} 2^{\frac{n+1}{2}}+1, &\mbox{if } n\geq 21 \ \text{is odd and } \B = \cdot \,;\\
2^{\frac{n}{2}+1}-3, &\mbox{if } n\geq 18 \ \text{is even and} \ \mathcal{B} = \cdot\,; \\
2^{\frac{n}{2}+1}-2, &\mbox{if } n\geq 2 \ \text{is even and} \ \mathcal{B} = \Hy. \end{cases}
\]

As highlighted in \cite{AhmMoh}, the quantity $\mathcal{S}_{3,2}(2,n,\cdot) $ may be interpreted as the size of a maximally large family $\mathcal{F}$ of non-empty subsets of an $n$-element set such that among every three distinct elements of  $\mathcal{F}$, there is a pair of sets whose intersection is of even cardinality. The values above confirm \cite[Conjecture 16]{AhmMoh} of Ahmadi and Mohammadian for sufficiently large, even $n$. Although, for odd $n$, we have shown that the relevant value is larger by one than what was conjectured. The sufficiently large $n$ assumption cannot be removed, see Remark~\ref{rem:2}.

We also prove an upper bound for $\mathcal{S}_{k,2}(q, n, \B)$. This is the finite field analogue of another question of Erd\H{o}s~\cite{AlSz,FuSt}.

\begin{theorem}
\label{thm:k2OS}
Let $q$ be an odd prime power and $k\geq 2$ be an integer. Suppose that $S\subset \F_q^n\setminus\{\bm{0}\}$ is $(k, 2)$-orthogonal with respect to a non-degenerate symmetric bilinear form $\B$. Then
$$
|S| \leq \bigg(k-1 + \frac{(k-1)^2}{q-k+1}\bigg)(q^{n/2}+1).
$$
\end{theorem}

Theorem~\ref{thm:k2OS} implies
\[
\mathcal{S}_{k,2} \leq (k-1+ o_{q \to \infty}(1)) q^{n/2},
\]
which is asymptotically sharp in some cases. For example for odd $q$, $k=n=4$ and $\varepsilon(\B)=1$ the union of the following three pairwise disjoint orthogonal sets has size $3 (q^2 -1)$:
\[
S_1 = \{ (x_1, x_1, x_2, x_2 ) : x_1, x_2 \in \F_q \} \setminus \{\bm{0}\}
\]
and
\[
S_2 = \{ (x_1, -x_1, x_2, -x_2 ) : x_1, x_2 \in \F_q \} \setminus \{\bm{0}\}
\]
and
\[
S_3 = \{ (x_1, x_2, x_2, -x_1 ) : x_1, x_2 \in \F_q \} \setminus \{\bm{0}\}.
\]

\subsection{Outline of the proofs of Theorems~\ref{thm:32OS odd} and~\ref{thm:q2even}}

The proofs of Theorems~\ref{thm:32OS odd} and~\ref{thm:q2even} are based on an inductive scheme developed by Ahmadi and Mohammadian~\cite{AhmMoh}. To outline the argument let us denote by $d_n=d_n(q,\B)$ the dimension of the largest orthogonal subspace of $\F_q^n$ with respect to a non-degenerate symmetric bilinear form $\B$. Both theorems take the form
\begin{equation}
\label{eqn:f}
\S_{3,2}(q,n, \B) = 2 q^{d_n} + f(q,n,\B),
\end{equation}
where $f(q,n,\B) \in \{-3,-2,1,4\}$. We show this by proving by induction a weaker statement of the form $\S_{3,2}(q,n, \B) \leq (2+o(1)) q^{d_n}$. Note here that the $o(1)$ term is for $q \to \infty$ for odd $q$ and $n \to \infty$ for $q=2$. A technical difficulty in carrying out the induction is that one must ensure that the restriction of $\B$ to the orthogonal complement considered is non-degenerate. 

The inductive argument that proves the weaker bound is based on the basic observation that for every $\bm{v} \in S$, the set of elements in $S \setminus\{\bm{v}\}$ not orthogonal to $\bm{v}$ constitute an orthogonal set. The structure of orthogonal sets has been determined by Berlekamp, and Ahmadi and Mohammadian~\cite{Ber,AhmMoh}. A key feature is that they contain few elements not in a single orthogonal subspace. We make repeated use of this fact. The second basic fact we use to our advantage for large odd $q$ is that a proper subspace of a vector space is significantly smaller than the vector space. 

Once the weak bound has been established, it is used to determine $f(q,n,\B)$. It is at this point that different arguments must be used according to $(q,n,\B)$. The key observation, implicit in the literature, is that if an orthogonal set contains just a few elements that are not self-orthogonal, then it is much smaller than $q^{d_n}$ (for large $q$). 

For $q=2$, which is not large, slightly different arguments are utilised. The basic fact that drives the proof is that, unlike for odd $q$, the set of vectors not orthogonal to any $\bm{v}$ contains at most half of any orthogonal subspace.  

The proofs of Theorems~\ref{thm:32OS odd} and \ref{thm:q2even} probably yield a characterisation of nearly extremal sets: they are mostly contained in two disjoint orthogonal subspaces of maximum dimension.

\section{Preparations}

We begin with some basic facts about Ramsey numbers. Given positive integers $s,t$ the \emph{Ramsey number} $R(s,t)$ is the least integer with the property that every graph on $R(s,t)$ vertices either contains a $K_s$ or the complement of the graph contains a $K_t$. We will use the following bounds:
\begin{equation}
\label{eqn:Ramsey}
R(3,3) = 6, R(3,4) =9, R(s,t) \leq \binom{s+t-2}{s-1}.
\end{equation}
The connection between Ramsey numbers and almost orthogonal sets goes back to at least the paper of Deaett~\cite{Dea}. A connection with other similar questions of Erd\H{o}s is detailed in~\cite{NeRo}. We deduce the following elementary observation concerning graphs. 

\begin{lemma} 
\label{lem: graph theory}
A triangle-free graph with the property that its complement is also triangle-free is either the 5-cycle or has at most 4 vertices.
\end{lemma}  

\begin{proof}
We denote by $H$ the graph. Its order is at most $R(3,3)-1 =5$ (by~\eqref{eqn:Ramsey}).

Suppose now that the order of $H$ is 5. Note that $H$ does not have a vertex of degree at least 3. This is because if such a vertex existed, then either two of its neighbours would be connected by an edge, giving rise to a triangle; or none of its neighbours would be connected by an edge, giving a triangle in the complement. Similarly, the complement  has no edge of degree at least 3. Therefore all vertices of the graph have degree 2. Finite graphs of constant degree 2 contain a cycle. The graph $H$ has no 3-cycle. It does not contain a 4-cycle (the fifth vertex would be isolated). Therefore it contains a 5-cycle, which is the entire graph.
\end{proof}

We proceed with results concerning orthogonal sets. The most important is a structural characterisation of orthogonal sets which we take from \cite[Lemma~3]{AhmMoh} -- see also~\cite{Ber} for $q=2$.
\begin{lemma}
 \label{claim:Struct}
    Let $\mathcal{B}$ denote a non-degenerate, symmetric bilinear form over $\F_q^n$, where $q$ is a prime power and $n\geq 2$. Suppose that $S\subset \F_q^n$ is an orthogonal set. Then, there exists an orthogonal subspace $V \subset  \{\bm{x}\in \F_q^n:\mathcal{B}(\bm{x}, \bm{x})=0\}$ and a set $T= \{\bm{x}\in S:\mathcal{B}(\bm{x}, \bm{x})\not=0\}$, such that $S\subset V\sqcup T$ and $2 \dim(V) + |T| \leq n.$
    \end{lemma}
    
Next, we recall \cite[Theorem 4]{AhmMoh}, which relying mainly on Lemma~\ref{claim:Struct}, obtains the following sharp bound on orthogonal sets. We note that in \cite{AhmMoh} $\mathcal{S}_{2,2}(q,n, \B)$ is denoted by $\mathcal{S}_0(q,n)$.
\begin{lemma}
\label{lem:OSUB}
For $n\geq 2$ and a prime power $q\geq 3$, let $\mathcal{B}$ denote a non-degenerate symmetric bilinear form over $\F_q^n$. Then
$$
\S_{2,2}(q,n,\B) = \begin{cases} q^{\frac{n-1}{2}}, &\mbox{if } n \ \text{is odd}; \\
q^{\frac{n}{2}}-1, &\mbox{if } n \ \text{is even and}\ \varepsilon(\mathcal{B}) = 1;\\ q^{\frac{n}{2}-1}+ 1, &\mbox{if } n \ \text{is even and}\ \varepsilon(\mathcal{B}) = \gamma .\end{cases}
$$
\end{lemma}

An analogue of Lemma~\ref{lem:OSUB}, for $q= 2$, was earlier proved by Berlekamp~\cite{Ber}. Also see \cite{Grav}.

\begin{lemma}
\label{lem:Berl}
\[
    \S_{2,2}(2,n,\cdot) =  \begin{cases}
    n,  &\mbox{if } n\leq 5;\\
    1+2^{\frac{n-1}{2}},  &\mbox{if } n\ \text{is odd and}\ n\geq 7;\\
    2^{\frac{n}{2}},  &\mbox{if } n\ \text{is even and}\ n\geq 6.
    \end{cases}
\]
\end{lemma}

The following corollary is central to our considerations.

\begin{lemma}
\label{claim:NsOrth}
For $n\geq 2$ and a prime power $q$, let $S\subset\F_q^n$ be a $(3, 2)$-orthogonal set with respect to a non-degenerate symmetric bilinear form $\B$ over $\F_q^n$. For $\bm{s}\in S$, define
   \begin{equation}
       \label{eqn:Nsdef}
  S_{\bm{s}} = \{\bm{x}\in S\setminus\{\bm{s}\}:\ \mathcal{B}(\bm{x}, \bm{s})\not= 0\}.
   \end{equation}
If $|S_{\bm{s}}| \geq 2$, then $S_{\bm{s}}$ is an orthogonal set. In particular, we may write for all $\bm{s}\in S$
\begin{equation*}
    S_{\bm{s}} = R_{\bm{s}} \sqcup T_{\bm{s}},
\end{equation*}
where $T_{\bm{s}} = \{\bm{x}\in S_{\bm{s}}:\mathcal{B}(\bm{x}, \bm{x})\not=0\}$ and $\langle R_{\bm{s}}\rangle = V_{\bm{s}}$, an orthogonal subspace of $\F_q^n$ that contains only self-orthogonal vectors.
\end{lemma}
\begin{proof}
Given two distinct vectors $\bm{x_1}, \bm{x_2}\in S_{\bm{s}}$, by the $(3,2)$-orthogonality of $S$, two of $\{\bm{x_1}, \bm{x_2}, \bm{s}\}$ must be mutually orthogonal. Thus, given the definition of $S_{\bm{s}}$, we must have $\mathcal{B}(\bm{x_1}, \bm{x_2}) = 0$. The rest follows from Lemma~\ref{claim:Struct} or is immediate when $|S_{\bm{s}}| \leq 1$.
\end{proof}

We also collect some basic facts about orthogonal subspaces as follows.
\begin{lemma} 
\label{lem: basic}
Let $V\subset \F_q^n$ denote an orthogonal subspace with at least three elements.
\begin{enumerate}[(i)]
\item\label{item:lem: basic:i} Every vector in $V$ is self-orthogonal. 
\item\label{item:lem: basic:ii} Suppose that $V$ is of maximum dimension and $V = \langle R\rangle$ for some $R\subset \F_q^n$. If $\bm{z} \notin V$ is a self-orthogonal vector, then $\bm{z}$ is not orthogonal to $R$.
\end{enumerate}
\end{lemma}

\begin{proof}
For the first statement, let $\bm{x}$ be any element of $V$ and $\bm{y}$ some other element of $V$. It follows that $\bm{x} + \bm{y} \in V \setminus\{\bm{x}\}$ and so 
$$
0 = \B(\bm{x}, \bm{x} + \bm{y}) =  \B(\bm{x}, \bm{x}) +  \B(\bm{x}, \bm{y}) = \B(\bm{x}, \bm{x}).
$$

For the second statement, we have $\B(\bm{z}, \bm{z}) = 0$. Suppose for a contradiction that $\bm{z} \perp R$. Then $\bm{z} \perp V$. Note that by the first part, for all $\lambda, \mu \in \F_q$ and $\bm{x}, \bm{y} \in V$ we have
\[
\B(\lambda \bm{z} + \bm{x}, \mu \bm{z}+ \bm{y})= \lambda \mu \B(\bm{z}, \bm{z}) +  \lambda \B(\bm{z}, \bm{y}) + \mu \B(\bm{z}, \bm{x}) + \B(\bm{x}, \bm{y})  = 0.
\]
Hence $\langle \{\bm{z}\} \cup R \rangle$ is an orthogonal subspace that strictly contains $V$, a contradiction.
\end{proof}

The next result forms the basis of the induction argument in the proof of Theorem~\ref{thm:32OS odd}. The key is to show that if we restrict $\B$ to a certain type of orthogonal complement, then it remains non-degenerate; and that, under a further condition, the equivalence class of $\B$ is conserved.   

\begin{lemma} 
\label{lem: induction}
Let $n \geq 2$, $\mathcal{B}$ be a non-degenerate symmetric bilinear form over $\F_q^n$, and $\{ \bm{v}, \bm{w}\} \subset \F_q^n$ be linearly independent. 
\begin{enumerate}[(i)]
\item\label{item:lem:induction:i} If
\[
\B(\bm{v}, \bm{w})^2 \neq \B(\bm{v}, \bm{v}) \B(\bm{w}, \bm{w}),
\]
then the restriction $\mathcal{B}\rst{\{\bm{v},\bm{w}\}^\perp}$ of $\mathcal{B}$ to the orthogonal complement of $\{\bm{v},\bm{w}\}$, is a non-degenerate symmetric bilinear form. 
\item\label{item:lem:induction:ii} If $q$ is odd, $\bm{v}$ and $\bm{w}$ are not mutually orthogonal, and $\bm{w}$ is self-orthogonal (that is $\B(\bm{v}, \bm{w}) \neq0$ and $\B(\bm{w}, \bm{w})=0$), then $\varepsilon(\mathcal{B} \rst{\{\bm{v},\bm{w}\}^\perp}) = \varepsilon(\mathcal{B})$.
\item\label{item:lem:induction:iii} If $q=2$, $n$ is even, $\bm{v}$ and $\bm{w}$ are not mutually orthogonal, and both $\bm{v},\bm{w}$ are self-orthogonal (that is $\B(\bm{v}, \bm{w}) \neq0$ and $\B(\bm{v}, \bm{v})=\B(\bm{w}, \bm{w})=0$), then $\B$ is equivalent to $\Hy$ if and only if $\mathcal{B} \rst{\{\bm{v},\bm{w}\}^\perp}$ is equivalent to $\Hy$.
\end{enumerate}
\end{lemma}  

\begin{proof}
Throughout the proof we write
\[
a = \B(\bm{v},\bm{v}) , b = \B(\bm{v},\bm{w}), c = \B(\bm{w},\bm{w}).
\]

For \eqref{item:lem:induction:i}, we first show $\langle \bm{v} ,\bm{w} \rangle \cap \{\bm{v},\bm{w}\}^\perp = \{\bm{0}\}$ and therefore that $\F_q^n = \langle \bm{v} ,\bm{w} \rangle \oplus \{\bm{v},\bm{w}\}^\perp$. Suppose $\lambda \bm{v} + \mu \bm{w} \in \{\bm{v},\bm{w}\}^\perp$. Applying $\B(\bm{v},-)$ and then $\B(\bm{w}, -)$ to both sides gives the linear system
\[
\left\{ \begin{matrix} \lambda a + \mu b = 0 \\ \lambda b + \mu c = 0 \end{matrix}\right..
\]
It follows that $\lambda=\mu=0$ because $ac \neq b^2$. 

We write $M_1$ for the matrix of $\mathcal{B}\rst{\langle \bm{v},\bm{w}\rangle}$ with respect to the basis $\{\bm{v},\bm{w}\}$, $M_2$ for the matrix of $\mathcal{B}\rst{\{\bm{v},\bm{w}\}^\perp}$ with respect to any basis, and $M$ for the matrix of $\mathcal{B}$ with respect to the union of these two bases, then
\[
M =  \begin{pmatrix} M_1 & 0 \\ 0 & M_2\end{pmatrix}.
\]
It follows immediately that if $\mathcal{B}$ is non-degenerate, then so is $\mathcal{B}\rst{\{\bm{v},\bm{w}\}^\perp}$.

For \eqref{item:lem:induction:ii}, we have $b \neq 0$ and $c=0$. We show
\[
M_1 =  \begin{pmatrix} a & b \\ b & 0\end{pmatrix} \sim \begin{pmatrix} 1 & 0 \\ 0 & -1\end{pmatrix},
\]
which proves $\varepsilon(\mathcal{B} \rst{\{\bm{v},\bm{w}\}^\perp}) = \varepsilon(\mathcal{B})$.

Let $\alpha, \beta, \gamma$ be solutions to $\alpha^2 - \gamma^2 = a$ and $\beta(\alpha - \gamma) =b$. Such $\alpha, \gamma$ exist because every element of $\F_q$ is the difference of two squares. The characteristic is not 2, so we can always take $\alpha \neq \gamma$ (even when $a =0$). Then there exists a suitable $\beta$. Now a simple calculation confirms
\[
\begin{pmatrix} \alpha & \beta \\ \gamma & \beta \end{pmatrix}^T \begin{pmatrix} 1 & 0 \\ 0 & -1 \end{pmatrix} \begin{pmatrix} \alpha & \beta \\ \gamma & \beta \end{pmatrix} = \begin{pmatrix} \alpha^2 - \gamma^2 & \beta(\alpha-\gamma) \\ \beta (\alpha-\gamma) & 0 \end{pmatrix}  = \begin{pmatrix} a & b \\ b & 0 \end{pmatrix};
\]
and $\det \begin{pmatrix} \alpha & \beta \\ \gamma & \beta \end{pmatrix} = \beta(\alpha- \gamma) = b \neq0$.

For \eqref{item:lem:induction:iii}, we have $b =1$ and $a=c=0$. Therefore $M_1 =  H$. So $\B$ is equivalent to $\Hy$ if and only if $\mathcal{B} \rst{\{\bm{v},\bm{w}\}^\perp}$ is equivalent to $\Hy$.
\end{proof}

\begin{remark}
\label{rem:1}
The condition in part \eqref{item:lem:induction:i} of Lemma~\ref{lem: induction} is necessary. Take, for example, $n=4$ and $\B$ the bilinear form given by the diagonal matrix with diagonal entries $(1,-1,1,-1)$. $\B$ is the dot product when $q=2$. Take $\bm{v} = (1,0,0,0)$ and $\bm{w} = (1,1,1,0)$. These are two linearly independent vectors with the numbers $a,b,c$ defined in the proof of the lemma all equal to 1. Hence $ac=b^2$. It is not true that $\langle \bm{v},\bm{w}\rangle$ trivially intersects $\{\bm{v},\bm{w}\}^\perp$ because the span of $\bm{w}-\bm{v} = (0,1,1,0)$ lies in both subspaces. Furthermore, $\B$ restricted to $\{\bm{v},\bm{w}\}^\perp$ is degenerate because $\bm{w}-\bm{v}$ is orthogonal to both $\bm{w}-\bm{v}$ and $(0,0,0,1)$, which span $\{\bm{v},\bm{w}\}^\perp$. 
\end{remark}

The next step is to bound the number of vectors in any $(3,2)$-orthogonal subset in $\F_q^n$ that are not self-orthogonal. It may be true that, analogously to the results of Rosenfeld and Deaett~\cite{Rose,Dea}, there are at most $2n$ such vectors. We prove a weaker result that suffices for our purposes. As part of the proof, we require a straightforward adaptation of~\cite[Proposition~4.4]{Dea}, which we state. The proof is nearly identical to that in \cite{Dea}.
\begin{lemma}
\label{lem:DeaP44}
Let $n \geq 1$ be a positive integer, $F$ be a field, and $S\subset F^n$ be a $(3, 2)$-orthogonal set with respect to a symmetric bilinear form. If $B\subset S$ is an orthogonal basis for $F^n$, then $S\setminus B$ is an orthogonal set.
\end{lemma}

We state and prove another result that is implicit in~\cite[Section 4]{Dea}. It is convenient to phrase many of the subsequent arguments in terms of the simple graph $G$ with vertex set $S$ and edges given by pairs of elements of $S$ that are not mutually orthogonal ($\bm{xy}$ is an edge precisely when $\B(\bm{x},\bm{y}) \neq 0$).

\begin{lemma} 
\label{lem: D}
Let $n \geq 1$ be a positive integer, $F$ a field, and $D\subset F^n$ be a $(3, 2)$-orthogonal set with respect to a symmetric bilinear form. If $D$ consists entirely of vectors that are not self-orthogonal, then
\[
|D| \leq \max\{ 2n , R(3,n) -1 \}\overset{\eqref{eqn:Ramsey}} = \begin{cases} 2n, &\mbox{if } 0 \leq n \leq 4 ;\\
\tfrac{n(n+1)}{2}-1, &\mbox{if } n\geq 5 . \end{cases}
\]
\end{lemma}

\begin{proof}
We use the graph $G$ described just above the statement of the lemma. The claim is true for $n=0$. For $n \geq1$ we observe that an independent set of vertices is an orthogonal set in $F^n$ and so is linearly independent (we need here that all vectors in $D$ are not self-orthogonal). If $G$ has an independent set $B$ of size $n$, then that set is linearly independent and therefore is a basis for $F^n$. By Lemma~\ref{lem:DeaP44} we get that $D \setminus B$ is orthogonal and hence contains at most $n$ elements. Hence $|D| = |B| + |D\setminus B| \leq 2n$. If $G$, which is triangle-free, has no independent set of size $n$, then $|D| < R(3,n)$, by the definition of $R(3,n)$.
\end{proof}

Note that by work of Ajtai, Koml\'os and Szemer\'edi, and of Kim \cite{AKS,Kim}
\[
R(3,n) = (1+ o_{n \to \infty}(1)) \frac{n^2}{\log n},
\]
with stronger explicit upper bounds in \cite{She}. This means that $|D| = o(n^2)$. 

We also extract this consequence of Lemma~\ref{claim:Struct} and Lemma~\ref{lem:DeaP44} from the proof of \cite[Theorem~17]{AhmMoh}.

\begin{lemma} 
\label{lem: AM}
Let $S \subset \F_q^n$ be a $(3,2)$-orthogonal set with respect to a non-degenerate symmetric bilinear form $\B$. If every pair of linearly independent vectors in $S$ is mutually orthogonal (that is $\B(\bm{x},\bm{y}) =0$ for every linearly independent $\{\bm{v}, \bm{w}\} \subset S$), then $|S| \leq \S_{2,2}(q,n,\B) + n$.
\end{lemma}

As the final result of this section, we recall \cite[Theorem~17]{AhmMoh}, which is a quantitatively weaker version of Theorem~\ref{thm:32OS odd}. We will use this result in Section~\ref{sec: character} and so provide a proof which follows the same scheme as that introduced in \cite{AhmMoh}, while paying special attention to certain intricacies involved in carrying out the induction. In particular, the proof relies on Lemma~\ref{lem: induction} to sidestep a potential issue that appears to have been overlooked in the original proof of \cite[Theorem~17]{AhmMoh}. It also serves as a prelude to the proof of Theorem~\ref{thm:32OS odd}. 

We employ for  the first of many times a decomposition of a $(3,2)$-orthogonal set $S$ that appears in \cite{AhmMoh}, and so we describe it in detail. Given two distinct elements $\bm{x}, \bm{y} \in S$, every element of $S \setminus \{\bm{x}, \bm{y}\}$ is either not orthogonal to $\bm{x}$, or not orthogonal to $\bm{y}$, or orthogonal to both $\bm{x}$ and $\bm{y}$. Using the notation of Lemma~\ref{claim:NsOrth} we decompose $S$ as follows
\begin{equation}
\label{eqn:1stdecomp}
S = S_{\bm{x}} \cup S_{\bm{y}} \cup S_{\bm{x}\bm{y}} \cup \{\bm{x}, \bm{y}\},
\end{equation}
where $S_{\bm{x}}$ and $S_{\bm{y}}$ are defined in~\eqref{eqn:Nsdef} and $S_{\bm{x}\bm{y}} = S \cap \{\bm{x}, \bm{y}\}^\perp$. Note that $ \{\bm{x}, \bm{y}\}$ can be left out if $\B(\bm{x}, \bm{y}) \neq 0$ because $\bm{x} \in S_{\bm{y}}$ and vice versa. When bounding $|S_{\bm{x}\bm{y}}|$ by induction it is essential that $\B \rst{\{\bm{x},\bm{y}\}^\perp}$ is non-degenerate.

\begin{proposition}
\label{prop:AM32UB}
Let $q$ be odd and let $\B$ be a non-degenerate symmetric bilinear form over $\F_q^n$. If $S\subset \F_q^n\setminus\{\bm{0}\}$ is $(3,2)$-orthogonal, then
\[
|S| \leq 3q^{\lfloor \frac{n}{2}\rfloor}. 
\]
\end{proposition}

\begin{proof}
We proceed by induction on $n$. Note that the result is true for $n\in \{0, 1\}$ because $|S| \leq 2$ and assume it is also true for all dimensions strictly less than $n$.

If every linearly independent pair of vectors in $S$ is mutually orthogonal, by Lemma~\ref{lem: AM} and Lemma~\ref{lem:OSUB}, we have
\[
|S| \leq \S_{2,2}(q,n,\B) + n \leq 2q^{\lfloor \frac{n}{2}\rfloor}.
\]
Hence suppose there exists a linearly independent pair $\{\bm{x}, \bm{y}\}\subset S$, with $\B(\bm{x}, \bm{y})\not=0$. If at least one of these vectors is self-orthogonal, by Lemma~\ref{lem: induction}~\eqref{item:lem:induction:i}, $\B \rst{\{\bm{x},\bm{y}\}^\perp}$ is non-degenerate. Recalling the decomposition \eqref{eqn:1stdecomp} and noting that $\bm{x}\in S_{\bm{y}}$ and $\bm{y}\in S_{\bm{x}}$, we have
\[
|S| \leq |S_{\bm{x}}| + |S_{\bm{y}}|+ |S_{\bm{xy}}|.
\]
Since $\bm{x}$ and $\bm{y}$ are linearly independent, $\{\bm{x}, \bm{y}\}^\perp$ constitutes a subspace of $\F_q^n$ of dimension $n-2$. Then, using that $S_{\bm{xy}}\subset \{\bm{x}, \bm{y}\}^\perp$, for $n\in \{2, 3\}$ we have $|S_{\bm{xy}}|\leq q^{n-2}\leq q$ and for $n\geq 4$ we have $|S_{\bm{xy}}|\leq 3q^{\lfloor \frac{n}{2}\rfloor-1}$ by the induction hypothesis. Furthermore, by Lemma~\ref{lem:OSUB} and Lemma~\ref{claim:NsOrth}, we have $|S_{\bm{x}}|, |S_{\bm{y}}| \leq q^{\lfloor \frac{n}{2}\rfloor}$. Adding this all up, we obtain the required result in this case.

Next, suppose that neither $\bm{x}$ nor $\bm{y}$ is self-orthogonal and note that in this case, we may no longer assume $\B \rst{\{\bm{x},\bm{y}\}^\perp}$ is non-degenerate (see Remark~\ref{rem:1}). If every pair of elements of $S_{\bm{xy}}$ is mutually orthogonal, by Lemma~\ref{lem:OSUB}, we have $|S_{\bm{xy}}|\leq q^{\lfloor \frac{n}{2}\rfloor}$ and the required result follows.
 Hence suppose there exist $\bm{v}, \bm{w}\in S_{\bm{xy}}$, with $\B(\bm{v}, \bm{w})\not=0$. Again, if at least one of $\{\bm{v}, \bm{w}\}$ is self-orthogonal, we may repeat the arguments of the first case to obtain the required result. Thus assume otherwise. Consider the decomposition
 \begin{equation}
 \label{eqn:AMDecomp2}
 S = S_{\bm{x}}\cup S_{\bm{v}}\cup S_{\bm{xv}}\cup \{\bm{x}, \bm{v}\}.
 \end{equation}
 By Lemma~\ref{lem: induction}~\eqref{item:lem:induction:i}, $\B \rst{\{\bm{x},\bm{v}\}^\perp}$ is non-degenerate. Employing the notation of Lemma~\ref{claim:NsOrth}, note that $\bm{y}\in T_{\bm{x}}$ and $\bm{w}\in T_{\bm{v}}$. Suppose there exists $\bm{z}\in R_{\bm{v}}$, with $\B(\bm{y}, \bm{z}) \not=0$. Then by Lemma~\ref{lem: induction}~\eqref{item:lem:induction:i}, $\B \rst{\{\bm{y},\bm{z}\}^\perp}$ is non-degenerate and we may repeat the arguments of the first case, with $\bm{z}$ in place of $\bm{x}$, to obtain the required result. Otherwise, if $\bm{y}$ is orthogonal to $R_{\bm{v}}$, it follows that $R_{\bm{v}} \sqcup \{\bm{y}, \bm{w}\}$ is an orthogonal set, which by Lemma~\ref{claim:Struct}, implies that $\dim(V_{\bm{v}}) \leq \lfloor n/2\rfloor -1$. By a similar argument, we may assume $R_{\bm{x}} \sqcup \{\bm{y}, \bm{w}\}$ is an orthogonal set and that $\dim(V_{\bm{x}}) \leq \lfloor n/2\rfloor -1$. Furthermore note that $\bm{w}\in S_{\bm{xy}}$ and so $\bm{w}\not\in S_{\bm{x}}$ and similarly $\bm{y}\not \in S_{\bm{v}}$.
 
 For $n\in \{2, 3\}$, by Lemma~\ref{claim:Struct} and the above observations, we have $|S_{\bm{x}}\cup S_{\bm{v}}|\leq 2n-2$ and $|S_{\bm{xv}}|\leq q$. Thus going back to \eqref{eqn:AMDecomp2}, we get $|S|\leq 2n+q \leq 3q$ as required. For $n\geq 4$, we may again use Lemma~\ref{claim:Struct} to see $|S_{\bm{x}}\cup S_{\bm{v}}|\leq 2(q^{\lfloor n/2\rfloor -1}+3)-2$. Then
 \begin{align*}
 |S| &\leq (2q^{\lfloor n/2\rfloor -1} +4) + 3q^{\lfloor \frac{n}{2}\rfloor-1} + 2\\
 &=5q^{\lfloor n/2\rfloor -1} +6\\ &\leq 3q^{\lfloor n/2\rfloor},
 \end{align*}
for all $n\geq 4$ and $q\geq 3$.
 \end{proof}

\section{Proof of Theorem~\ref{thm:32OS odd}}

We set 
\[
d_n = \begin{cases} \frac{n-1}2, &\mbox{if } n\geq 3 \ \text{is odd};\\
\frac{n}2, &\mbox{if } n\geq 2 \ \text{is even and}\ \varepsilon(\mathcal{B}) = 1; \\
 \frac{n}2 -1, &\mbox{if } n\geq 2 \ \text{is even and}\ \varepsilon(\mathcal{B}) = \gamma. \end{cases}
\]
It was proved in \cite{Vinh} that $d_n$ is the dimension of the largest orthogonal subspace of $\F_q^n$ (also follows from Lemma~\ref{claim:Struct}). Note that $d_{n-2} = d_n -1$.

We proceed by induction on $n$. For $n=0$ and $n=1$ the size of the largest $(3, 2)$-orthogonal set is at most 2, and the theorem follows. 

We will show that either $|S| \leq q^{d_n} + O(q^{d_n-1})$ or that $S$ possesses certain properties that make proving the theorem a matter of case analysis. For sufficiently large $q$ the former upper bound is smaller than the one in the theorem.

We phrase the argument in terms of the graph $G$ with vertex set $S$ and two vectors adjacent precisely when they are not mutually orthogonal. The two properties of $G$ we use is that it is triangle free (follows from $S$ being $(3,2)$-orthogonal) and the largest independent set in $G$ having size at most $\mathcal{S}_{2,2}(q,n,\B)$ (because an independent set is an orthogonal subset of $\F_q^n$). We will also use the fact that every orthogonal set in $\F_q^n$ has size at most $\mathcal{S}_{2,2}(q,n,\B)$, a quantity that is determined in Lemma~\ref{lem:OSUB}.

By Lemma~\ref{lem: AM} we may assume from now on the existence of linearly independent $\{\bm{v}, \bm{w}\} \subset S$ with $\bm{vw}$ an edge (that is $\B(\bm{v},\bm{w}) \neq 0$). This is because $n \leq q^{d_n}-2$ for all $n\geq 2$ when $q\geq 5$. We decompose $S$ in the neighbourhood $S_{\bm{v}}$ of $\bm{v}$, the neighbourhood $S_{\bm{w}}$ of $\bm{w}$, and the set of vertices $S_{\bm{v}\bm{w}}$ that are not adjacent to either $\bm{v}$ or $\bm{w}$:
\[
S = S_{\bm{v}} \cup S_{\bm{w}} \cup S_{\bm{vw}},
\]
where $S_{\bm{vw}} \subset \{\bm{v},\bm{w}\}^\perp$. We follow the set up of Lemma~\ref{claim:NsOrth} and decompose $S_{\bm{v}} = R_{\bm{v}} \sqcup T_{\bm{v}}$ with $R_{\bm{v}}$ spanning the orthogonal vector space $V_{\bm{v}}$. 

When $n=2$ we get that $S_{\bm{vw}}$ is a subset of a zero dimensional vector space that does not include $\bm{0}$ and so is empty. Since both $S_{\bm{v}}$ and $S_{\bm{w}}$ are orthogonal sets, we get $|S| \leq 2 \mathcal{S}_{2,2}(q,2,\B)$. This proves the theorem for $n=2$.

For $n \geq 3$ we have to be more careful when dealing with $S_{\bm{vw}}$. We need the following to be able to apply the second part of Lemma~\ref{lem: induction}.

\begin{lemma}
\label{lem: self orth}
For $n\geq 3$, let $S \subset \F_q^n$ be a $(3,2)$-orthogonal set with respect to a non-degenerate symmetric bilinear form $\B$. If every pair of linearly independent self-orthogonal vectors in $S$ is orthogonal to one another (that is $\B(\bm{x},\bm{y}) =0$ for all linearly independent self-orthogonal $\bm{x}, \bm{y} \in S$), then 
\[
|S| \leq \begin{cases} \S_{2,2}(q,n,\B) + 2n, &\mbox{if } 0 \leq n \leq 4 ;\\
\S_{2,2}(q,n,\B) + \tfrac{n(n+1)}{2}-1, &\mbox{if } n\geq 5 . \end{cases}
\]
\end{lemma}

\begin{proof}
Let $D$ be the set of vectors in $S$, which are not self-orthogonal:
\[
D = \{ \bm{x} \in S : \B(\bm{x},\bm{x}) \neq 0\}.
\]
By the hypothesis on $S$ we have that $S \setminus D$ is an orthogonal set. Hence $|S \setminus D| \leq \S_{2,2}(q,n,\B)$. The claim follows by bounding $|D|$ via Lemma~\ref{lem: D}.
\end{proof}

The upper bound on $|S|$ in Lemma~\ref{lem: self orth} is smaller than the bound in Theorem~\ref{thm:32OS odd} for $n \geq 3$ when $q\geq 5$. From now on we assume the existence of linearly independent vectors $\{\bm{v},\bm{w}\}$ that are self-orthogonal but are not mutually orthogonal:
\[
\B(\bm{v},\bm{v}) = \B(\bm{w},\bm{w}) = 0 \text{, but } \B(\bm{v},\bm{w}) \neq 0.
\]

Recalling the definition of $f=f(q,n, \B)$ inferred from Theorem~\ref{thm:32OS odd} and~\eqref{eqn:f}, we get from Lemma~\ref{lem: induction}~\eqref{item:lem:induction:ii} 
\begin{equation} 
\label{eq:S*}
|S_{\bm{vw}}| \leq 2 q^{d_{n-2}} + f = 2 q^{d_n-1} + f.
\end{equation}
Therefore $|S_{\bm{vw}}|$ is much smaller than the bound on $|S|$ we are trying to prove. What drives the proof is that if either $V_{\bm{v}}$ or $V_{\bm{w}}$ is not of maximum dimension, then we are done. To see why, suppose $V_{\bm{v}}$ is not of maximum dimension. Then using Lemma~\ref{claim:Struct}, Lemma~\ref{lem:OSUB}, Lemma~\ref{claim:NsOrth} and \eqref{eq:S*} we get
\begin{align*}
|S| 
&  \leq |S_{\bm{v}}| + |S_{\bm{w}}| + |S_{\bm{vw}}| \\
& \leq (q^{d_n-1}+3) + (q^{d_n}+1) + (2 q^{d_n-1} +f) \\
& = q^{d_n} + 3 q^{d_n-1} + 4+f.
\end{align*} 
Since for $q\geq 7$ we have $3 q^{d_n-1} + 4 \leq q^{d_n}$, we assume from now on $\dim(V_{\bm{v}}) = \dim(V_{\bm{w}}) = d_n$. 

One more property we need is that $\bm{v} \in V_{\bm{w}}$ and $\bm{w} \in V_{\bm{v}}$. To confirm, say the latter, note that there is no edge from $\bm{w}$ to $R_{\bm{v}}$ (because the graph is triangle-free). Therefore $\bm{w} \perp \langle R_{\bm{v}} \rangle = V_{\bm{v}}$. By Lemma~\ref{lem: basic}~\eqref{item:lem: basic:ii}, and using the fact that $\bm{w}$ is self-orthogonal, we get $\bm{w} \in V_{\bm{v}}$.

We summarise all this in a proposition.

\begin{proposition}
\label{prop: summary}
Let $q \geq 7$ be an odd prime power, $n\geq 3$ and $S \subset \F_q^n$ be a $(3,2)$-orthogonal set with respect to a non-degenerate symmetric bilinear form. Then $|S|$ satisfies the upper bound of Theorem~\ref{thm:32OS odd} unless there exist linearly independent self-orthogonal vectors $\bm{v}$ and $\bm{w}$ with $\dim(V_{\bm{v}})=\dim(V_{\bm{w}}) = d_n$; and $\bm{v} \in V_{\bm{w}}$ and $\bm{w} \in V_{\bm{v}}$. In this case $S_{\bm{vw}} = S \cap \{\bm{v},\bm{w}\}^\perp$ satisfies $|S_{\bm{vw}}| \leq 2 q^{d_n-1} + f(q,n, \B)$, with $f(q,n, \B)$ inferred from Theorem~\ref{thm:32OS odd} and~\eqref{eqn:f}.
\end{proposition} 

The final preparatory result is that for the remaining $S$ described in Proposition~\ref{prop: summary}, $R_{\bm{z}}$ is considerably smaller than $q^{d_n}$ for all $\bm{z} \in S_{\bm{vw}}$. The proof is typical of forthcoming considerations. The key observation is that if a subspace of a vector space does not contain a single element of the vector space, then it is considerably smaller.

\begin{lemma}
\label{lem: vz}
Let $q$ be an odd prime power and let $S \subset \F_q^n$ be $(3,2)$-orthogonal with respect to a non-degenerate symmetric bilinear form $\B$. Suppose $\{\bm{v},\bm{w}\} \subset S$ is a linearly independent subset that consists of two self-orthogonal vectors that are not mutually orthogonal (that is $\B(\bm{v},\bm{v})=\B(\bm{w},\bm{w})=0$, but $\B(\bm{v},\bm{w})\neq 0$). If $\bm{z} \in S_{\bm{vw}} = S \cap \{\bm{v},\bm{w}\}^\perp$, then 
\[
|R_{\bm{z}}| \leq 3q^{d_n-1} -3.  
\]
\end{lemma}

\begin{proof}
We have
\[
S \subset (V_{\bm{v}} \setminus\{\bm{0}\}) \cup T_{\bm{v}} \cup (V_{\bm{w}} \setminus\{\bm{0}\}) \cup T_{\bm{w}} \cup S_{\bm{vw}}.
\]
By Lemma~\ref{claim:NsOrth} we know that $R_{\bm{z}}$ contains only self-orthogonal vectors and so is disjoint from $ T_{\bm{v}} \cup T_{\bm{w}}$. Hence
\[
|R_{\bm{z}}| \leq (|V_{\bm{v}} \cap V_{\bm{z}}|-1) + (|V_{\bm{w}} \cap V_{\bm{z}}|-1) + |R_{\bm{z}}\cap S_{\bm{vw}}|.
\] 
$V_{\bm{v}} \neq V_{\bm{z}}$ because $\bm{w} \in V_{\bm{v}} \setminus V_{\bm{z}}$. Hence $V_{\bm{v}} \cap V_{\bm{z}}$ is a proper subspace of $V_{\bm{v}}$ and is therefore not of maximum dimension. This means $|V_{\bm{z}} \cap V_{\bm{v}}| \leq q^{d_n-1}$. Similarly $|V_{\bm{z}} \cap V_{\bm{w}}| \leq q^{d_n-1}$. Moreover, note that $R_{\bm{z}}\cap S_{\bm{vw}}$ is an orthogonal subset of $\{\bm{v}, \bm{w}\}^{\perp}$, on which non-degeneracy and type of $\B$ is preserved by Lemma~\ref{lem: induction}~\eqref{item:lem:induction:ii}. Thus by Lemma~\ref{claim:Struct}, we have $|R_{\bm{z}}\cap S_{\bm{vw}}| \leq q^{d_n-1}-1$. Putting everything together gives the desired bound.
\end{proof}

We begin the final stage of the proof of the theorem. We assume we are in the remaining case detailed in Proposition~\ref{prop: summary}. Let 
\[
S_{\bm{vw}}^* = S_{\bm{vw}}\setminus (V_{\bm{v}}\cup V_{\bm{w}}).
\]
We distinguish between two different cases.

{\bf Case 1:} An edge exists between $R_{\bm{v}} \cup R_{\bm{w}}$ and $S_{\bm{vw}}^*$.

Suppose $\bm{uz}$ is an edge with $\bm{z} \in S_{\bm{vw}}$ and, say, $\bm{u} \in R_{\bm{v}}$. Our first claim is that $\{\bm{u},\bm{z}\}$ is linearly independent. Indeed if $\bm{z} = \lambda \bm{u}$, then we would have $\B(\bm{z}, \bm{u}) = \lambda B(\bm{u},\bm{u}) =0$, which contradicts $\bm{uz}$ being an edge. Furthermore, $\bm{u}$ is self-orthogonal and so by Lemma~\ref{lem: induction}~\eqref{item:lem:induction:ii} we get that $\B \rst{\{\bm{u}, \bm{z}\}^\perp}$ is non-degenerate and $\varepsilon(\B)$ is preserved. 

We have
\[
|S| \leq |S_{\bm{u}}| + |R_{\bm{z}}| + |T_{\bm{z}}| + |S_{\bm{uz}}|.
\]
We have the following bounds: by Lemma~\ref{lem:OSUB} and Lemma~\ref{claim:NsOrth} $|S_{\bm{u}}|\leq q^{d_n} +1$; by Lemma~\ref{claim:Struct} and Lemma~\ref{lem: vz} (and its proof), $|R_{\bm{z}}| + |T_{\bm{z}}| \leq 3q^{d_n-1}+1$; and by induction, just like in~\eqref{eq:S*}, $|S_{\bm{uz}}| \leq 2q^{d_n-1}+f$. In total
\[
|S| \leq q^{d_n} + 5 q^{d_n-1}+2+f.
\]
We are done because for $q\geq 7$ and $n \geq 3$, $5 q^{d_n-1}+2 \leq q^{d_n}$.

{\bf Case 2:} No edge exists between $R_{\bm{v}} \cup R_{\bm{w}}$ and $S_{\bm{vw}}^*$. 

There is no edge from $S_{\bm{vw}}^*$ to $R_{\bm{v}}$ and therefore $S_{\bm{vw}}^*$ is orthogonal to $R_{\bm{v}}$. It follows that $S_{vw}^*$ is orthogonal to $V_{\bm{v}} = \langle R_{\bm{v}} \rangle$. Similarly, $S_{vw}^*$ is orthogonal to $V_{\bm{w}}$. All vectors in $T_{\bm{v}} \cup T_{\bm{w}} \cup S_{\bm{vw}}$ are not self-orthogonal by Lemma~\ref{lem: basic}~\eqref{item:lem: basic:ii} and $\dim(V_{\bm{v}})$ being maximum. We use the decomposition
\begin{equation}
\label{eqn:filan step}
S \subset (V_{\bm{v}}\setminus \{\bm{0}\}) \cup (V_{\bm{w}}\setminus \{\bm{0}\}) \cup (T_{\bm{v}} \cup T_{\bm{w}} \cup S_{\bm{vw}}). 
\end{equation}
We consider the three different possibilities separately.

\underline{Even $n\geq 4$ and $\varepsilon(\B) =1$}. Our aim is to show that $T_{\bm{v}} = T_{\bm{w}} = S_{\bm{vw}}^* = \emptyset$. Then by \eqref{eqn:filan step}
\[
|S| \leq (|V_{\bm{v}}| -1) + (|V_{\bm{w}}|-1) \leq 2 (q^{d_n}-1) = 2 q^{n/2} -2.
\]
We may assume $T_{\bm{v}}=T_{\bm{w}} = \emptyset$ else, by Lemma~\ref{claim:Struct}, $V_{\bm{v}}$ or $V_{\bm{w}}$ are not of maximum dimension, which is not allowed by Proposition~\ref{prop: summary}. To show that $S_{\bm{vw}}^* = \emptyset$, suppose for a contradiction that $\bm{z} \in S_{\bm{vw}}^*$, then $V_{\bm{v}} \cup \{\bm{z}\}$ would be an orthogonal set, forcing, via Lemma~\ref{claim:Struct},  $V_{\bm{v}}$ not to have maximum dimension. 

\underline{Odd $n\geq 3$}. We want to show $|T_{\bm{v}}| + |T_{\bm{w}}| + |S_{\bm{vw}}^*| \leq 3$. Then by \eqref{eqn:filan step}
\[
|S| \leq (|V_{\bm{v}}| -1) + (|V_{\bm{w}}|-1) +3 \leq 2 q^{d_n}+1 = 2 q^{(n-1)/2} +1.
\]
For any distinct $\bm{x},\bm{y} \in S_{\bm{vw}}^*$, $\bm{xy}$ is an edge. This is because if $\bm{xy}$ were not an edge, then $V_{\bm{v}} \cup \{\bm{x},\bm{y}\}$ would be an orthogonal set of size $|V_{\bm{v}}|+2$, which, by Lemma~\ref{claim:Struct}, would force $V_{\bm{v}}$ not to have maximum dimension. Therefore the induced subgraph on $S_{\bm{vw}}^*$ is complete and triangle-free. Hence $S_{\bm{vw}}^*$ must have at most two vertices. Moreover, by Lemma~\ref{claim:Struct}, $|T_{\bm{v}}|, |T_{\bm{w}}| \leq 1$ (else the subspaces do not have maximum dimension). We are done unless $|S_{\bm{vw}}^*|=|T_{\bm{v}} \cup T_{\bm{w}}|=2$. Suppose $S_{\bm{vw}}^* = \{\bm{x},\bm{y}\}$ with $\bm{xy}$ an edge, and $T_{\bm{v}} = \{\bm{u}\}$. The graph is triangle-free and so one of $\bm{ux}$, $\bm{uy}$ is not an edge. Suppose that $\bm{ux}$ is not an edge. Then $V_{\bm{v}} \cup \{\bm{u},\bm{x}\}$ is an orthogonal set, forcing $V_{\bm{v}}$ not to be of maximal dimension.

\underline{Even $n\geq 4$ and $\varepsilon(\B) =\gamma$}. We want to show $|T_{\bm{v}}| + |T_{\bm{w}}| + |S_{\bm{vw}}^*| \leq 6$. Then  by \eqref{eqn:filan step}
\[
|S|  \leq (|V_{\bm{v}}| -1) + (|V_{\bm{w}}|-1) +6 \leq 2 q^{d_n}+4 = 2 q^{n/2-1} +4.
\]
In fact, by Lemma~\ref{claim:Struct} and Proposition~\ref{prop: summary}, $|T_{\bm{w}}| \leq 2$ and we must show $|T_{\bm{v}}| + |S_{\bm{vw}}^*| \leq 4$.

Note that every vertex in $T_{\bm{v}} \cup S_{\bm{vw}}^*$ is orthogonal to $R_{\bm{v}}$ and therefore is orthogonal to $V_{\bm{v}}$. Similarly, $S_{\bm{vw}}^*$ is orthogonal to $\langle V_{\bm{v}} \cup V_{\bm{w}} \rangle = V_{\bm{v}} + V_{\bm{w}}$. Now consider the graph $H$ induced on $T_{\bm{v}} \cup S_{\bm{vw}}^*$. This is a triangle-free graph. Moreover, it has no independent set of size 3 because otherwise we could join this set to $V_{\bm{v}}$ and obtain an orthogonal set of size $|V_{\bm{v}}| +3$, which would force $V_{\bm{v}}$ not to have maximum dimension. By Lemma~\ref{lem: graph theory} we get $|T_{\bm{v}} \cup S_{\bm{vw}}^*| \leq 5$ with equality only when $H$ is a 5-cycle. Our final task is to rule out this possibility. Suppose for a contradiction that $H$ is a 5-cycle.

We set $T_{\bm{v}} = \{\bm{u}_1, \bm{u}_2\}$ and $S_{\bm{vw}}^* = \{\bm{z}_1, \bm{z}_2, \bm{z}_3\}$. $\bm{u}_1\bm{u}_2$ is not an edge (because both $\bm{u}_1, \bm{u}_2$ are incident to $\bm{v}$) and so $H$ can be taken to be the 5-cycle $\bm{z}_1 \bm{u}_1 \bm{z}_2 \bm{u}_2 \bm{z}_3$. The complement of $H$ is the 5-cycle $\bm{z}_1 \bm{z}_2 \bm{z}_3\bm{u}_1 \bm{u}_2$. In the complement of $H$ vertex adjacency is equivalent to orthogonality, and so $\{\bm{z}_1 ,\bm{z}_2\}\subset (V_{\bm{v}} + V_{\bm{w}})^\perp$ is an orthogonal (and hence linearly independent) set in the orthogonal complement of $ V_{\bm{v}} + V_{\bm{w}}$.

We make a small digression to investigate the dimension of $V_{\bm{v}} + V_{\bm{w}}$. We may assume $V_{\bm{v}} \cap V_{\bm{w}} = \{\bm{0}\}$. This is because if the intersection is non-trivial, then
\[
|S| \leq |(V_{\bm{v}} \cup V_{\bm{w}}) \setminus \{\bm{0}\}| + |T_{\bm{v}}| + |T_{\bm{w}}| + |S_{\bm{vw}}^*| \leq (2 q^{d_n} -q) +7 \leq 2 q^{d_n} +4,
\]
and we are done. We may therefore assume that $\dim(V_{\bm{v}} + V_{\bm{w}}) = 2 d_n = n-2$ and hence
\[
\F_q^n = (V_{\bm{v}} + V_{\bm{w}}) \oplus \langle \bm{z}_1 \rangle \oplus \langle \bm{z}_2\rangle.
\]

To complete the argument we exploit the orthogonality relations encoded in the complement of $H$ and show that $\bm{z_3} = \bm{0}$, the contradiction we are after. To start note that $\bm{z}_3 \in  \langle \bm{z}_1 , \bm{z}_2\rangle $. As $\bm{z}_3$ is orthogonal to $\bm{z}_2$, we get $\bm{z}_3 = \lambda \bm{z}_1$. We are left to show $\lambda=0$. 

We next show $\bm{u}_1, \bm{u}_2 \in V_{\bm{v}} \oplus \langle \bm{z}_1 , \bm{z}_2\rangle$.  Let's start with, say, the decomposition
\[
\bm{u}_1 = \alpha \bm{z}_1 + \beta \bm{z}_2 + \bm{x}_{\bm{v}} +  \bm{x}_{\bm{w}},
\] 
for $\alpha, \beta \in \F_q$, $\bm{x}_{\bm{v}} \in V_{\bm{v}}$, and $\bm{x}_{\bm{w}} \in V_{\bm{w}}$. Suppose for a contradiction that $\bm{x}_{\bm{w}} \neq \bm{0}$. By Lemma~\ref{lem: basic}~\eqref{item:lem: basic:ii} and the maximality of $\dim(V_{\bm{v}})$ we get that the self-orthogonal vector $\bm{x}_{\bm{w}}$ is not orthogonal to $V_{\bm{v}}$. Therefore there exists $\bm{y} \in V_{\bm{v}} \setminus \{\bm{0}\}$ such that $\B(\bm{x}_{\bm{w}}, \bm{y}) \neq 0$. But then
\[
\B(\bm{u}_1, \bm{y}) = \B(\bm{x}_{\bm{w}}, \bm{y}) \neq 0,
\]
a contradiction to $\bm{u}_1$ being orthogonal to the whole of $V_{\bm{v}}$. We therefore have
\[
\bm{u}_1 = \alpha \bm{z}_1 + \beta \bm{z}_2 + \bm{x}_{\bm{v}} \text{ and } \bm{u}_2 = \alpha' \bm{z}_1 + \beta' \bm{z}_2 + \bm{x}_{\bm{v}}'.
\]
 
Now, $\bm{u}_2$ is orthogonal to $\bm{z_1}$ and so $0 = \alpha' \B(\bm{z}_1, \bm{z}_1)$, which gives $\alpha '=0$. Moreover $\bm{u}_2 \notin V_{\bm{v}}$, which gives $\beta' \neq 0$. Next, $\bm{u}_2$ is orthogonal to $\bm{u}_1$ and so $0= \beta \beta' \B(\bm{z}_2, \bm{z}_2)$. Hence $\beta=0$ and, similarly to above, $\bm{u}_1 = \alpha \bm{z}_1 + \bm{x}_{\bm{v}}$ for $\alpha \in \F_q^*$. Finally, $\bm{u}_1$ is orthogonal to $\bm{z}_3 = \lambda \bm{z}_1$. Hence $0 = \lambda \alpha \B(\bm{z}_1,\bm{z}_1)$, which implies the desired $\lambda =0$.

The graph $H$ is therefore not a 5-cycle and consequently $|T_{\bm{v}} \cup S_{\bm{vw}}^*| \leq 4$. The proof of the theorem is concluded.

\section{Character sum proof of Theorem~\ref{thm:32OS odd} for even $n$, $\varepsilon(\B) =1$ and all odd $q$}
\label{sec: character}
First, we recall some basic facts from the theory of character sums, which we use to give an alternative proof of Theorem~\ref{thm:32OS odd} for even $n$ and $\varepsilon(\B) =1$ that holds for all odd $q$. See, for example, \cite[Chapter~5]{LiNi} for more details.

\begin{lemma}
\label{claim:orthsg}
Let $H$ be a subgroup of a finite abelian group $G$ and $\chi$ a character of $G$, then
     \begin{equation*}
\sum_{g\in H}\chi(g)=  \begin{cases}
    |H| \quad &\text{if}\quad \chi\ \text{is trivial on}\ H,\\
    0 \quad &\text{otherwise}.
    \end{cases}
\end{equation*}
\end{lemma}

Let $e_p(x) = \exp(2\pi i x/p)$, $\mathrm{Tr}(x) = x + x^p +\dots+x^{p^{m-1}}$ (recalling $q=p^m$) and $\psi(x) = e_p(\mathrm{Tr}(x))$. Then the functions $\{\psi(\lambda x):\lambda\in \F_q\}$ determine all of the characters of $\F_q$. 
\begin{lemma}
\label{calim:ntchar}
Let $\mathcal{B}$ denote a non-degenerate, symmetric bilinear form over $\F_q^n$, where $q$ is odd and $n\geq 2$. Let $V$ denote a subspace of $\F_q^n$. Suppose that $\bm{s}\not\in V^{\perp}$. Then $\psi(\mathcal{B}(\bm{s}, -))$ is a nontrivial character of $V$.
\end{lemma}
The next result is a slight extension of Vinogradov's bound on bilinear character sums, which appears, for example, in \cite[p. 92]{Vino}. Also see \cite[Lemma~5]{Shp} for the special case, where $\B$ is the dot product.
\begin{lemma}
\label{claim:VinBiSum}
Given $X, Y\subset \F_q^n$, we have
\begin{equation*}
    \bigg|\sum_{\bm{x}\in X}\sum_{\bm{y}\in Y} \psi\Big( \mathcal{B}(\bm{x}, \bm{y})\Big)\bigg| \leq \sqrt{|X||Y|q^n}.
\end{equation*}
\end{lemma}
\begin{proof}
We apply the triangle inequality and then the Cauchy-Schwarz inequality to get
 \begin{align*}
    \bigg|\sum_{\bm{x}\in X}\sum_{\bm{y}\in Y} \psi\Big(  \mathcal{B}(\bm{x}, \bm{y})\Big)\bigg|^2 &\leq |X| \sum_{\bm{x}\in X}\bigg|\sum_{\bm{y}\in Y} \psi\Big( \mathcal{B}(\bm{x}, \bm{y})\Big)\bigg|^2\\ 
    &\leq |X| \sum_{\bm{x}\in \F_q^n}\bigg|\sum_{\bm{y}\in Y} \psi\Big( \mathcal{B}(\bm{x}, \bm{y})\Big)\bigg|^2\\ 
    &= |X| \sum_{\bm{x}\in \F_q^n}\sum_{\bm{y}, \bm{z}\in Y} \psi\Big( \mathcal{B}(\bm{x}, \bm{y}-\bm{z})\Big)\\ 
    &= |X| \sum_{\bm{y}, \bm{z}\in Y}\sum_{\bm{x}\in \F_q^n} \psi\Big( \mathcal{B}(\bm{x}, \bm{y}-\bm{z})\Big)\\ 
    &=|X||Y|q^n.
\end{align*}
To obtain the last equality, we used the fact that the inner sum in the penultimate line equals $q^n$ if $\bm{y} = \bm{z}$ and zero otherwise. This, in turn, follows from the observation that ${\F_q^n}^{\perp} = \{\bm{0}\}$, combined with Lemmas~\ref{claim:orthsg} and \ref{calim:ntchar}.
\end{proof}

\begin{proof}[Proof of Theorem~\ref{thm:32OS odd} for even $n$ and $\varepsilon(\B)=1$]

Firstly, replace $S$ by $S\sqcup \{\bm{0}\}$. This makes calculations easier. We will take away $\bm{0}$ at the end of the proof. For $\bm{s}\in S$, write $S^{'}_{\bm{s}} = \{\bm{x}\in S:\mathcal{B}(\bm{x}, \bm{s})\not=0\}$. Also write $D = \{\bm{x}\in S: \mathcal{B}(\bm{x}, \bm{x})\not=0$\}. Recalling \eqref{eqn:Nsdef}, note that
$$
\sum_{\bm{s}\in S}|S^{'}_{\bm{s}}| = \sum_{\bm{s}\in S}|S_{\bm{s}}| + |D|.
$$
Now
 \begin{align}
 \label{eqn:1slb}
     |S|^{2}-\sum_{\bm{s}\in S}|S_{\bm{s}}| - |D| &= |S|^{2}-\sum_{\bm{s}\in S}|S^{'}_{\bm{s}}|\\ \nonumber &= \bigg|\sum_{\bm{s}_1\in S}\sum_{\bm{s}_2\in S\setminus S^{'}_{\bm{s}_1}}\psi(\mathcal{B}(\bm{s}_1, \bm{s}_2))\bigg|.
    \end{align}
    Here, we used just that for $\bm{s}_1\in S$ and $\bm{s}_2\in S\setminus S^{'}_{\bm{s}_1}$, we have $\psi(\mathcal{B}(\bm{s}_1, \bm{s}_2)) = 1$.
    By the triangle inequality, we also have
     \begin{align*}
    \bigg|\sum_{\bm{s}_1\in S}\sum_{\bm{s}_2\in S\setminus S^{'}_{\bm{s}_1}}\psi(\mathcal{B}(\bm{s}_1, \bm{s}_2))\bigg| &\leq \bigg|\sum_{\bm{s}_1\in S}\sum_{\bm{s}_2\in S}\psi(\mathcal{B}(\bm{s}_1, \bm{s}_2))\bigg|+ \bigg|\sum_{\bm{s}_1\in S}\sum_{\bm{s}_2\in S^{'}_{\bm{s}_1}}\psi(\mathcal{B}(\bm{s}_1, \bm{s}_2))\bigg| \\
    & \leq \bigg|\sum_{\bm{s}_1\in S}\sum_{\bm{s}_2\in S}\psi(\mathcal{B}(\bm{s}_1, \bm{s}_2))\bigg|+ \bigg|\sum_{\bm{s}_1\in S}\sum_{\bm{s}_2\in S_{\bm{s}_1}}\psi(\mathcal{B}(\bm{s}_1, \bm{s}_2))\bigg| +|D|.
 \end{align*}
Next, let $W = \{\bm{x}\in S: |S_{\bm{x}}| \geq 2\}$. Then, using Lemmas~\ref{claim:Struct} and \ref{claim:NsOrth}, we have
  \begin{align}
  \label{eqn:DSWNs}
      \bigg|\sum_{\bm{s}_1\in W}\sum_{\bm{s}_2\in S_{\bm{s}_1}}\psi(\mathcal{B}(\bm{s}_1, \bm{s}_2))\bigg| &\leq \bigg|\sum_{\bm{s}_1\in W}\sum_{\bm{s}_2\in (V_{\bm{s}_1}\sqcup T_{\bm{s}_1})}\psi(\mathcal{B}(\bm{s}_1, \bm{s}_2))\bigg| \\ \nonumber &+ \bigg|\sum_{\bm{s}_1\in W}\sum_{\bm{s}_2\in (V_{\bm{s}_1}\sqcup T_{\bm{s}_1})\setminus S_{\bm{s}_1}}\psi(\mathcal{B}(\bm{s}_1, \bm{s}_2))\bigg| .
  \end{align}
  To bound the first sum, on the RHS of \eqref{eqn:DSWNs}, we first apply the triangle inequality to obtain
  \begin{align*}
    \bigg|\sum_{\bm{s}_1\in W}\sum_{\bm{s}_2\in (V_{\bm{s}_1}\sqcup T_{\bm{s}_1})}\psi(\mathcal{B}(\bm{s}_1, \bm{s}_2))\bigg| &\leq \bigg|\sum_{\bm{s}_1\in W}\sum_{\bm{s}_2\in V_{\bm{s}_1}}\psi(\mathcal{B}(\bm{s}_1, \bm{s}_2))\bigg| + \bigg|\sum_{\bm{s}_1\in W}\sum_{\bm{s}_2\in T_{\bm{s}_1}}\psi(\mathcal{B}(\bm{s}_1, \bm{s}_2))\bigg|.
  \end{align*}
 Note that for $\bm{s}\in W$, it follows from the definition of $V_{\bm{s}}$ (see Lemma~\ref{claim:NsOrth}) that $\bm{s}\not\in V_{\bm{s}}^{\perp}$. Thus, by Lemma~\ref{calim:ntchar}, $\psi(\mathcal{B}(\bm{s}, -))$ constitutes a character of $\F_q^n$ which is nontrivial on the subspace $V_{\bm{s}}$. Consequently, by Lemma~\ref{claim:orthsg}, we have
 \[
 \sum_{\bm{x}\in V_{\bm{s}}}\psi(\mathcal{B}(\bm{s}, \bm{x})) = 0 \quad\implies\quad \sum_{\bm{s}_1\in W}\sum_{\bm{s}_2\in V_{\bm{s}_1}}\psi(\mathcal{B}(\bm{s}_1, \bm{s}_2)) =0.
 \]
 Then, based on this observation and applications of the triangle inequality, we obtain 
 \begin{align*}
    \bigg|\sum_{\bm{s}_1\in W}\sum_{\bm{s}_2\in (V_{\bm{s}_1}\sqcup T_{\bm{s}_1})}\psi(\mathcal{B}(\bm{s}_1, \bm{s}_2))\bigg| &\leq  \bigg|\sum_{\bm{s}_1\in W}\sum_{\bm{s}_2\in T_{\bm{s}_1}}\psi(\mathcal{B}(\bm{s}_1, \bm{s}_2))\bigg|\\
    &\leq \sum_{\bm{s}_1\in W}\bigg|\sum_{\bm{s}_2\in T_{\bm{s}_1}}\psi(\mathcal{B}(\bm{s}_1, \bm{s}_2))\bigg|\\
    &\leq \sum_{\bm{s}\in W}|T_{\bm{s}}|.
  \end{align*}
  The second sum, on the RHS of \eqref{eqn:DSWNs}, is bounded trivially by
  \begin{equation*}
      \sum_{\bm{s}\in W} |V_{\bm{s}}| + |T_{\bm{s}}| - |S_{\bm{s}}|.
  \end{equation*}
  Going back to \eqref{eqn:1slb}, we have
  \begin{align*}
      |S|^{2} \leq 2|D| + \bigg|\sum_{\bm{s}_1\in S}\sum_{\bm{s}_2\in S}\psi(\mathcal{B}(\bm{s}_1, \bm{s}_2))\bigg| + \sum_{\bm{s}\in W} |V_{\bm{s}}| + 2|T_{\bm{s}}|+ \sum_{\bm{s}\in S\setminus{W}} 2|S_{\bm{s}}|.
  \end{align*}
  By Lemma~\ref{claim:VinBiSum}, we know
  $$
  \bigg|\sum_{\bm{s}_1\in S}\sum_{\bm{s}_2\in S}\psi(\mathcal{B}(\bm{s}_1, \bm{s}_2))\bigg| \leq |S|q^{n/2}.
  $$
  For $\bm{s}\in W$, write $k_{\bm{s}} = \dim(V_{\bm{s}})$. Then by Lemma~\ref{claim:Struct}, we know $|V_{\bm{s}}| + 2|T_{\bm{s}}| \leq q^{k_{\bm{s}}}+2n-4k_{\bm{s}} \leq q^{n/2}.$
  Furthermore, for $\bm{s}\in S\setminus{W}$, we have $2|S_{\bm{s}}|\leq q^{n/2}$.
  So adding it all up,
  $$
  |S|^{2} \leq 2|S|q^{n/2} + 2|D|.
  $$
  At this stage we go back to the original $S$ that does not include $\bm{0}$. The above becomes
  \begin{equation}
     \label{eqn:SpUBID}
   |S| \leq 2q^{n/2} +\Big\lfloor\frac{2|D|}{|S|}\Big\rfloor-1.
  \end{equation}

Note that from \eqref{eqn:SpUBID}, one can only deduce the bound $|S|\leq 2q^{n/2} + 1$. However, we proceed to sharpen this bound through an analysis of the set $D$. Some aspects of the remaining arguments can certainly be simplified if we are not aiming to prove the theorem for all $q$. To deal with some small technicalities that follow, we require the bound
\begin{equation}
    \label{eqn:n=2C}
    |S|\leq 2(q-1),
\end{equation}
for $n=2$, all odd $q$ and either scenarios $\varepsilon(\B) \in \{ 1 , \gamma\}$. This bound has already been established as the base case of the induction in the proof of Theorem~\ref{thm:32OS odd}\footnote{In the proof of Theorem~\ref{thm:32OS odd}, when applying Lemma~\ref{lem: AM}, to avoid a lengthy multi-case analysis, it is assumed that $q>3$. However one may easily confirm that, for our purposes here, the argument remains valid when $q=3$.}. In particular, henceforth assume $n\geq 4$. 

First we establish
\begin{equation}
\label{eqn:impb2}
    |S|\leq 2q^{n/2}-1,
\end{equation}
which follows from \eqref{eqn:SpUBID} if $2|D|< |S|$. So suppose otherwise. Using the bound on $|D|$ provided by Lemma~\ref{lem: D}, we get $|S|\leq 16$ for $n=4$ and $|S|\leq n(n+1)-2$ with both being better than \eqref{eqn:impb2} for all $q\geq 3$ and $n\geq 4$.

It remains to show \eqref{eqn:impb2} may be lowered by one. To this end, we consider a few cases showing that the assumption $|S| = 2q^{n/2}-1$ leads to contradictions on either the size or the parity of $|S|.$ In particular, the following observations will be useful.

\begin{claim}
\label{claim:SmDp}
Let $S\subset \F_q^n$ denote a maximal $(3, 2)$-orthogonal set. Then $|S\setminus D| \equiv 0 \pmod{q-1}$.
\end{claim}
\begin{proof}
Each $\bm{v}\in S\setminus D$ is self-orthogonal and so $S$ must contain the entire punctured line $l_{\bm{v}} = \{\lambda\bm{v}:\lambda\in \F_q^{*}\}$, otherwise maximality of $S$ is violated. The result follows noting that $l_{\bm{v}}\cap l_{\bm{w}}$ is empty if $\bm{w}\not\in l_{\bm{v}}$ and of size $q-1$ otherwise.
\end{proof}

\begin{claim}
\label{lem:Dv2}
At least one of the following statements holds.
\begin{enumerate}[(i)]
\item\label{item:lem:Dv2:i} $|S|\leq 2q^{n/2}-2$,
\item\label{item:lem:Dv2:ii} $|D| = 2$,
\item\label{item:lem:Dv2:iii} $|T_{\bm{v}}|\leq 1$ for each $\bm{v}\in S$, where $T_{\bm{v}} = D\cap S_{\bm{v}}$.
\end{enumerate}
\end{claim}

\begin{proof}
We suppose that neither \eqref{item:lem:Dv2:ii} nor \eqref{item:lem:Dv2:iii} is true and prove \eqref{item:lem:Dv2:i}. Thus, assume there exist $\bm{v}\in S$, distinct elements $\bm{w}_1, \bm{w}_2\in D\setminus\{\bm{v}\}$, with $\B(\bm{v}, \bm{w}_1)$ and $\B(\bm{v}, \bm{w}_2)$ both non-zero, and potentially a fourth element $\bm{w}_3 \in D$, which need not be distinct from the previous ones. We consider two main cases as to whether $\bm{v}$ is self-orthogonal or not. 

First, assume $\B(\bm{v}, \bm{v}) \not= 0$, which in particular implies $\bm{v}\in T_{\bm{w}_1}\cap T_{\bm{w}_2}$. By the $(3,2)$-orthogonality of $S$, we have $\B(\bm{w}_1, \bm{w}_2) = 0$ and so firstly $\{\bm{w}_1,\bm{w}_2\}$ is linearly independent and secondly by Lemma~\ref{lem: induction}~\eqref{item:lem:induction:i}, $\mathcal{B}\rst{\{\bm{w}_1,\bm{w}_2\}^\perp}$ is non-degenerate. Write
\[
S = S_{\bm{w}_1} \cup S_{\bm{w}_2} \cup S_{\bm{w}_1\bm{w}_2}\cup \{\bm{w}_1,\bm{w}_2\}.
\]
For $(q, n) = (3, 4)$, by \eqref{eqn:n=2C}, we have
\[
|S| \leq 2(q^{n/2 -1}+1)-1 +2(q-1) + 2 =   13 < 16  = 2q^{n/2}-2.
\]
For other admissible choices of $(q, n)$, by Proposition~\ref{prop:AM32UB}, we have
\[
|S| \leq 2(q^{n/2-1}+1) -1 + 3q^{n/2-1} + 2 =5q^{n/2-1}+3 < 2q^{n/2} -2.
\]

Next, assume $\B(\bm{v}, \bm{v}) = 0$. In this case $\bm{w}_3\not\in \{\bm{v}, \bm{w}_1, \bm{w}_2\}.$ We split this case further by first assuming that $\bm{w}_3$ is orthogonal to $R_{\bm{v}}$ (using the notation of Lemma~\ref{claim:NsOrth}). This implies $R_{\bm{v}}\sqcup \{\bm{w}_1, \bm{w}_2, \bm{w}_3\}$ is an orthogonal set. Further note that $\{\bm{w}_1, \bm{v}\}$ is linearly independent and that by parts \eqref{item:lem:induction:i} and \eqref{item:lem:induction:ii} of Lemma~\ref{lem: induction}, $\mathcal{B}\rst{\{\bm{w}_1,\bm{v}\}^\perp}$ is non-degenerate and its equivalence class is preserved. Write
\[
S = S_{\bm{w}_1} \cup S_{\bm{v}} \cup S_{\bm{w}_1\bm{v}}.
\]
For $(q, n) = (3, 4)$, by Lemma~\ref{claim:Struct} and \eqref{eqn:n=2C}, we have
\[
|S| \leq (q^{n/2}-1) + (q^{n/2-2}+3) + 2q^{n/2-1}-2 = 16  = 2q^{n/2}-2.
\]
For the remaining combinations of $(q, n)$, we use the bound \eqref{eqn:impb2} to get
\[
|S|\leq (q^{n/2}-1) + (q^{n/2-2}+3) + 2q^{n/2-1}-1 \leq 2q^{n/2}-2.
\]

Finally, suppose there exists some $\bm{u}\in R_{\bm{v}}$, such that $\B(\bm{w}_3, \bm{u})\not= 0$. By definition $\B(\bm{u}, \bm{u}) = 0$ and $\B(\bm{u}, \bm{v})\not= 0$, from which we may deduce $\{\bm{u}, \bm{v}\}$ is linearly independent and that by parts \eqref{item:lem:induction:i} and \eqref{item:lem:induction:ii} of Lemma~\ref{lem: induction}, $\mathcal{B}\rst{\{\bm{u},\bm{v}\}^\perp}$ is non-degenerate and its equivalence class is preserved. Write
\[
S = S_{\bm{u}} \cup S_{\bm{v}} \cup S_{\bm{u}\bm{v}}
\]
and note that $\bm{w}_1, \bm{w}_2\in T_{\bm{v}}$ and $\bm{w}_3\in T_{\bm{u}}$. 

We use the bound \eqref{eqn:impb2} to get
\[
|S|\leq 2(q^{n/2-1}+1) + 2q^{n/2-1}-1 = 4q^{n/2-1}+1 \leq 2q^{n/2}-2
\]
for all odd $q$ and $n \geq 4$.
\end{proof}

 \begin{claim}
 \label{claim:hpub}
Suppose that $S\subset \{\bm{z}\}^{\perp}$, where $\bm{z}\in \F_q^n$ is not self-orthogonal. Then 
\[
|S| \leq 3q^{n/2 -1} \leq 2q^{n/2} -2,
\]
for all $n\geq 4$ and odd $q$.
\end{claim}
\begin{proof}
Since $\bm{z}$ is not self-orthogonal, the restriction of $\B$ on $\{\bm{z}\}^{\perp}$ remains non-degenerate. Now, using that $S= S\cap \{\bm{z}\}^{\perp}$, we may use Proposition~\ref{prop:AM32UB}, to obtain the required result.
\end{proof}

Writing $Q = \{\bm{v}\in S: \bm{v}\in (S\setminus\{\bm{v}\})^{\perp}\}$, note that \begin{equation*}
    S= \bigcup_{\bm{v}\in S}S_{\bm{v}} \cup Q.
\end{equation*}
Now if $\bm{z}\in D\cap Q$, then $S\setminus\{\bm{z}\}\subset \{\bm{z}\}^{\perp}$. Thus, by Claim~\ref{claim:hpub}, we have
\[
|S| \leq 1 + 3q^{n/2 -1} \leq 2q^{n/2} -2,
\]
for all $n\geq 4$ and odd $q$. In particular, we may assume
\begin{equation}
    \label{eqn:DvCoversD}
D= \bigcup_{\bm{v}\in S}T_{\bm{v}}.
\end{equation}

Suppose $|S| = 2q^{n/2}-1$. Note that, if $|D|$ is even, by Claim~\ref{claim:SmDp} we must have that $|S|$ is even leading to a contradiction. Recalling Claim~\ref{lem:Dv2}, if statement \eqref{item:lem:Dv2:ii} holds, we are done and so assume statement \eqref{item:lem:Dv2:iii} is true. Let $\bm{w}\in D$ and so, by \eqref{eqn:DvCoversD}, we have $T_{\bm{v}} = \{\bm{w}\}$ for some $\bm{v}\in S$. If $\bm{v}$ is self-orthogonal, then firstly $\{\bm{v}, \bm{w}\}$ is linearly independent and secondly by parts \eqref{item:lem:induction:i} and \eqref{item:lem:induction:ii} of Lemma~\ref{lem: induction}, $\mathcal{B}\rst{\{\bm{v},\bm{w}\}^\perp}$ is non-degenerate and its equivalence class is unchanged. We write 
\[
S = S_{\bm{v}} \cup S_{\bm{w}} \cup S_{\bm{v} \bm{w}}
\]
and use Lemma~\ref{claim:Struct} as before, being mindful of the crucial fact that $S_{\bm{v}}$ contains exactly one non-self-orthogonal element. For $(q, n) = (3, 4)$, we have, by \eqref{eqn:n=2C}
\[
|S| \leq  q^{n/2-1} + (q^{n/2}-1)+ 2q^{n/2-1}-2 = 3q^{n/2-1} + q^{n/2}-3 = 15<16 = 2q^{n/2}-2.
\]
For other combinations of $(q, n)$, by \eqref{eqn:impb2}, we have
\[
|S| \leq q^{n/2-1} + (q^{n/2}-1) + 2q^{n/2-1}-1 = 3q^{n/2-1} + q^{n/2}-2\leq 2q^{n/2}-2.
\]
Both bounds above contradict the presumed size of $S$. Then, we must have that $\bm{v}$ is not self-orthogonal. It follows that $T_{\bm{w}} = \{\bm{v}\}$, which in turn implies that elements of $D$ occur in pairs. As explained above, this contradicts the presumed parity of $|S|$, concluding the proof.

\end{proof}

\section{Proof of Theorem~\ref{thm:q2even}}
 
 The proof of Theorem~\ref{thm:q2even} is similar to the proof of Theorem~\ref{thm:32OS odd}. The differences arise from having characteristic 2 (the theory of bilinear forms is different) and not being able to assume that $q$ is large enough. We are however free to assume that $n$ is large enough. A fact special to $\F_2^n$ that we use is that every two distinct non-zero vectors are linearly independent. In particular the requirement for $\{\bm{v}, \bm{w}\}$ to be linearly independent in Lemma~\ref{lem: induction} becomes redundant.

From now on we use the notation in Lemma~\ref{claim:Struct} and Lemma~\ref{claim:NsOrth}. The following simple inequality will be useful. It is specific to $\F_2^n$, is true for all $n$, and is sharp.

\begin{lemma}
\label{lem:Rv}
For $n \geq 2$, let $S \subset \F_2^n$ be a $(3,2)$-orthogonal set with respect to a non-degenerate symmetric bilinear form $\B$. If $\bm{v} \in S$, then in the notation of Lemma~\ref{claim:NsOrth}, $|R_{\bm{v}}| \leq |V_{\bm{v}}|/2$.
 \end{lemma}

 \begin{proof}
 Note that $R_{\bm{v}}$ is disjoint from $R_{\bm{v}}+R_{\bm{v}}$. Indeed, if $\bm{x}, \bm{y} \in R_{\bm{v}}$, then $\B(\bm{v} , \bm{x}) = \B(\bm{v} , \bm{y})=1$. Therefore $\B(\bm{v} , \bm{x}+ \bm{y}) = 1+1=0$. This means that $\bm{x}+ \bm{y} \notin R_{\bm{v}}$.

Now, $V_{\bm{v}}$ is a vector space containing $R_{\bm{v}}$. Therefore $R_{\bm{v}}$ and $R_{\bm{v}} + R_{\bm{v}}$ are two disjoint sets contained in $V_{\bm{v}}$. Hence
\[
2 |R_{\bm{v}}| \leq |R_{\bm{v}}| + |R_{\bm{v}} + R_{\bm{v}}| \leq |V_{\bm{v}}|.\qedhere
\]
\end{proof}

We derive a bound on $|S_{\bm{v}}|$. 

\begin{lemma}
\label{lem:Sv}
For $n \geq 2$, let $S \subset \F_2^n$ be a $(3,2)$-orthogonal set with respect to a non-degenerate symmetric bilinear form $\B$. If $\bm{v} \in S$, then in the notation of Lemma~\ref{claim:NsOrth}:
\begin{itemize}
\item If $\B = \cdot$, then
\[
    |S_{\bm{v}}| \leq  \begin{cases}
    n,  &\mbox{if } n\leq 7;\\
    1+2^{\frac{n-1}{2}-1},  &\mbox{if } n\ \text{is odd and}\ n\geq 9;\\
    2^{\frac{n}{2}-1},  &\mbox{if } n\ \text{is even and}\ n\geq 8.
    \end{cases}
\]
\item If $\B = \Hy$ and $n$ is even, then $|S_{\bm{v}}| \leq 2^{\frac{n}{2}-1}$.
\end{itemize}
\end{lemma}

 \begin{proof}
We begin with $\B = \cdot$. By Lemma~\ref{claim:NsOrth} and Lemma~\ref{lem:Rv}, we have
\[
|S_{\bm{v}}| \leq |R_{\bm{v}}| + |T_{\bm{v}}| \leq \frac{ |V_{\bm{v}}|}{2} + |T_{\bm{v}}|.
\]
By Lemma~\ref{claim:Struct} we have $\dim(V_{\bm{v}}) \leq \lfloor(n - |T_{\bm{v}}|)/2 \rfloor$. Setting $t = |T_{\bm{v}}|$ we get
\[
|S_{\bm{v}}| \leq 2^{\lfloor\frac{n-t}{2} \rfloor-1} +t.
\]
A routine calculation confirms that, for  $n\leq 7$, the right side is maximum when $t=n$. Otherwise, the maximum is achieved when $t=1$ for odd $n$ and when $t=0$ for even $n$.

If $\B = \Hy$, there are no non-self-orthogonal vectors and so, similarly to above, $|S_{\bm{v}}| \leq |V_{\bm{v}}| /2 \leq 2^{\frac{n}{2}-1}$.
\end{proof}

We first prove the theorem for the hyperbolic form $\Hy$.

\begin{proposition}
\label{prop:Hy}
Let $n \geq 2$ be even and $S \subset \F_2^n$ be a $(3,2)$-orthogonal set with respect to the hyperbolic form $\Hy$. Then 
\[
    |S| \leq  2^{\frac{n}{2} +1} -2.
\]
\end{proposition}

\begin{proof}
We prove the claim by induction. For $n=2$, $\F_2^2 \setminus\{\bm{0}\}$ is not $(3,2)$-orthogonal. So the claim is true for $n=2$.

For the inductive step, we may assume there exist linearly independent $\bm{v}, \bm{w}$ such that $\bm{v} \cdot \bm{w}=1$. If not, then $S$ is an orthogonal set and by Lemma~\ref{claim:Struct}, we have the better bound $|S| \leq 2^{\frac{n}{2}} -1$. By Lemma~\ref{lem: induction}~\eqref{item:lem:induction:iii}, $\Hy \rst{\{\bm{v}, \bm{w}\}^\perp}$ is non-degenerate and is equivalent to $\Hy$ (in this lower dimensional vector space). By the induction hypothesis we have
\[
|S_{\bm{vw}}| \leq 2^{\frac{n}{2}} -2.
\]
Hence by Lemma~\ref{lem:Sv}, 
\[
|S| \leq |S_{\bm{v}}| + |S_{\bm{w}}| + |S_{\bm{vw}}| \leq 2^{\frac{n}{2}-1} + 2^{\frac{n}{2}-1} + (2^{\frac{n}{2}} -2) = 2^{\frac{n}{2}+1} -2. \qedhere 
\]
\end{proof}

From now on we mainly restrict our attention to the dot product, though we will use Proposition~\ref{prop:Hy} for even $n$ because we sometimes use Lemma~\ref{lem: induction}~\eqref{item:lem:induction:i} and the restriction of the dot product may be equivalent to $\Hy$. To prove the theorem we consider two cases separately depending on whether $S$ contains a vector that is not self-orthogonal or not. We first  prove the theorem when all vectors in $S$ are self-orthogonal. The proof is similar to that of Proposition~\ref{prop:Hy}. 

\begin{proposition}
\label{prop:q2evenI}
For $n\geq 1$, let $S \subset \F_2^n$ be a $(3,2)$-orthogonal set with respect to the dot product. If $S$ consists entirely of self-orthogonal vectors, then 
\[
    |S| \leq  \begin{cases}
     2^{\frac{n+1}{2}}-2,  &\mbox{if } n\ \text{is odd};\\
    2^{\frac{n}{2} +1} -3,  &\text{if } n\ \text{is even}.
    \end{cases}
\]
\end{proposition}

\begin{proof}
We prove the claim by induction. For $n=1$, $|S|=0$. For $n=2$, we have $S \subset \{(1,1)\}$ and the claim follows. 

For the inductive step, we may assume there exist linearly independent $\bm{v}, \bm{w}$ such that $\bm{v} \cdot \bm{w}=1$. If not, then $S$ is an orthogonal set and by Lemma~\ref{lem:Berl}, we have a better bound on $|S|$ than required. Furthermore, by Lemma~\ref{lem: induction}~\eqref{item:lem:induction:iii}, the dot product restricted to $\{\bm{v}, \bm{w}\}^\perp$ is equivalent to the (lower dimensional) dot product and $S_{\bm{vw}}$ contains only self-orthogonal vectors. 
For even $n$, by the induction hypothesis we have
\[
|S_{\bm{vw}}| \leq 2^{\frac{n}{2}} -3.
\]
All vectors in $S$ are self-orthogonal and so $T_{\bm{v}}=T_{\bm{w}} = \emptyset$. Therefore $S_{\bm{v}}=R_{\bm{v}}$. Lemma~\ref{lem:Rv} gives
\[
|R_{\bm{v}}| \leq 2^{\frac{n}{2}-1}.
\]
The same holds for $\bm{w}$. Putting everything together gives
\[
|S| \leq 2^{\frac{n}{2}-1} + 2^{\frac{n}{2}-1} + (2^{\frac{n}{2}} -3) = 2^{\frac{n}{2}+1} -3. 
\]

For odd $n$, the induction hypothesis gives
\[
|S_{\bm{vw}}| \leq 2^{\frac{n-1}{2}}-2.
\]
Again $T_{\bm{v}}=T_{\bm{w}} = \emptyset$ and so by Lemma~\ref{lem:Sv}, we have
\[
|S_{\bm{v}}|, |S_{\bm{w}}| \leq 2^{\frac{n-1}{2} -1}.
\]
This gives
\[
|S| \leq 2^{\frac{n-1}{2} -1} + 2^{\frac{n-1}{2} -1}+ (2^{\frac{n-1}{2}}-2) = 2^{\frac{n+1}{2}}-2,
\]
as required.
\end{proof}

The next step is to prove a bound for all $S$ that is weaker than that in Theorem~\ref{thm:q2even}. It will be used to prove the theorem when $S$ contains a vector that is not self-orthogonal.

\begin{lemma}
\label{lem:weak}
For $n \geq 1$, let $S \subset \F_2^n$ be a $(3,2)$-orthogonal set with respect to a non-degenerate symmetric bilinear form. Then  
\[
    |S| \leq  \begin{cases}
    2^{\frac{n+1}{2}} + 2n -2,  &\mbox{if } n=1,3;\\
    2^{\frac{n+1}{2}} + \tfrac{n(n+1)}{2}-3,  &\mbox{if } n\geq 5 \ \text{is odd};\\
    2^{\frac{n}{2}+1} +2n-3,  &\mbox{if } n = 2,4;\\
    2^{\frac{n}{2}+1} + \tfrac{n(n+1)}{2} -4,  &\mbox{if } n\geq 6\ \text{is even}.
    \end{cases}
\]
 \end{lemma}

\begin{proof}
If the bilinear form is equivalent to $\Hy$, the result follows from Proposition~\ref{prop:Hy}.

If the bilinear form is equivalent to the dot product, we let $D \subset S$ be the collection of vectors in $S$ that are not self-orthogonal. The claim follows by applying Proposition~\ref{prop:q2evenI} to $S \setminus D$ and Lemma~\ref{lem: D} to $D$.
\end{proof}

We continue with the case when there is a vector that is non-self-orthogonal. The proof is longer because we cannot initiate the induction (for example, Remark~\ref{rem:2} on p.~\pageref{rem:2} shows $\S_{3,2}(2,4,\cdot) \geq 7 > 2^3 -3$) and because we can no longer assume, say, $T_{\bm{v}} = \emptyset$.

\begin{proposition}
\label{prop:q2evenII}
Let $n$ be an integer and $S \subset \F_2^n$ be a $(3,2)$-orthogonal with respect to the dot product. If $S$ contains a vector that is not self-orthogonal, then 
\[
    |S| \leq  \begin{cases}
    2^{\frac{n+1}{2}} +1,  &\mbox{if } n\geq 21\ \text{is odd};\\
    2^{\frac{n}{2}+1} -3,  &\mbox{if } n\geq 18\ \text{is even}.
    \end{cases}
\]
\end{proposition}

\begin{proof}
We begin with familiar notation. We let $G$ be the simple graph with vertex set $S$ and edges given by not mutually orthogonal pairs of vertices, and 
\[
D = \{ \bm{v} \in S : \bm{v} \cdot \bm{v} =1 \}.
\]

\underline{Even $n$}. Let $\bm{z} \in D$. We consider two separate cases according to whether there exists an edge between $\bm{z}$ and $S \setminus D$ or not.

Suppose first that there is no edge between $\bm{z}$ and $S \setminus D$. Then $S\setminus D \subset \{\bm{z}\}^\perp$. The dot product restricted to $\{\bm{z}\}^\perp$ is non-degenerate (because $\bm{z}$ is not self-orthogonal). The dimension of $\{\bm{z}\}^\perp$ is odd and so the restriction is equivalent to the (lower dimensional) dot product. Moreover, all elements of $S\setminus D$ are self-orthogonal. Therefore by Proposition~\ref{prop:q2evenI} we get $|S \setminus D| \leq 2^{\frac{n}{2}} -2$. By Lemma~\ref{lem: D} we have $|D| \leq \tfrac{n(n+1)}{2}-1$. Hence (because $n \geq 14)$
\[
|S| \leq (2^{\frac{n}{2}} -2) + (\tfrac{n(n+1)}{2}-1) \leq 2^{\frac{n+1}{2}} -3.
\]

Next we suppose that there exists and edge $\bm{vz}$ with $\bm{v}\in S\setminus D$. We have 
\[
|S| \leq |S_{\bm{v}}| + |S_{\bm{z}}| + |S_{\bm{vz}}|.
\]
By Lemma~\ref{lem: induction}~\eqref{item:lem:induction:i}, the dot product restricted to $\{\bm{v}, \bm{z}\}^\perp$ is non-degenerate. Lemma~\ref{lem:weak} gives
\[
|S_{\bm{vz}}| \leq 2^{\frac{n}{2}} +\tfrac{(n-2)(n-1)}{2} -4.
\]
To bound $|S_{\bm{v}}|$ note $\bm{z} \in T_{\bm{v}}$. Writing $t = |T_{\bm{v}}|$, and applying Lemma~\ref{claim:NsOrth} and Lemma~\ref{lem:Rv} we get (using $n \geq 12$)
\[
|S_{\bm{v}}| \leq 2^{\lfloor \frac{n-1-t}{2} \rfloor -1}+ t \leq 2^{\frac{n}{2} -2} +1.
\]
By Lemma~\ref{lem:Sv} we get $|S_{\bm{z}}| \leq 2^{\frac{n}{2} -1}$.

Putting everything together gives (using $n \geq 18$)
\[
|S| \leq 2^{\frac{n}{2}+1} -3 - (2^{\frac{n}{2} -2} -\tfrac{(n-2)(n-1)}{2}) \leq 2^{\frac{n}{2} +1} -3.
\]
\underline{Odd $n$}.
If $D$ contains up to three elements, the required result follows from Proposition~\ref{prop:q2evenI}. Let $\bm{x}, \bm{y}, \bm{z} \in D$ denote three distinct elements and let $H$ be the graph induced on $\{\bm{x}, \bm{y}, \bm{z}\}$. The graph $H$ is not a triangle because $D$ is a subset of a $(3, 2)$-orthogonal set. We consider three cases based on the number of edges in $H$. 

Suppose $H$ is the empty graph. First, assume a pair of the sets $R_{\bm{x}}, R_{\bm{y}}, R_{\bm{z}}$ has non-empty intersection. Namely, say $\bm{v}\in R_{\bm{x}}\cap R_{\bm{y}}$. Consider the decomposition 
\begin{equation}
    \label{eqn:q2odddecomp}
S = S_{\bm{x}}\cup S_{\bm{v}}\cup S_{\bm{xv}}.
\end{equation}
Note that $\bm{x}, \bm{y}\in T_{\bm{v}}$ and that by Lemma~\ref{lem: induction}~\eqref{item:lem:induction:i}, we may apply Lemma~\ref{lem:weak} to obtain
\begin{equation}
    \label{eqn:q2odd1stUB}
|S| \leq (2^{\frac{n-1}{2}-2}+3) + (2^{\frac{n-1}{2}-1}+1) +(2^{\frac{n-1}{2}}+\tfrac{(n-2)(n-1)}{2}-3) \leq 2^{\frac{n+1}{2}}+1,
\end{equation}
for $n\geq 21$. Thus suppose $R_{\bm{x}}, R_{\bm{y}}, R_{\bm{z}}$ are pairwise disjoint. This means that $R_{\bm{z}} \cup \{\bm{x} , \bm{y}\}$ is orthogonal. Since $\bm{z}\in S_{\bm{xy}}$ and the dot product restricted to $\{\bm{x}, \bm{y}\}^\perp$ is non-degenerate (by Lemma~\ref{lem: induction}~\eqref{item:lem:induction:i}), Lemma~\ref{lem:Sv} gives 
\[
|S_{\bm{z}}|= |S_{\bm{z}}\cap S_{\bm{xy}}| \leq 1+2^{\frac{n-1}{2}-2}.
\]
Considering the decomposition 
\begin{equation}
 \label{eqn:q2odddecomp2}
S = S_{\bm{x}}\cup S_{\bm{z}}\cup S_{\bm{xz}}\cup \{\bm{x}, \bm{z}\},
\end{equation}
and using Lemma~\ref{lem: induction}~\eqref{item:lem:induction:i}, Lemma~\ref{lem:Sv} (as well as its proof), and Lemma~\ref{lem:weak}, we have
\begin{equation}
    \label{eqn:q2odd1stUB2}
|S| \leq (2^{\frac{n-1}{2}-1}+1)+ (2^{\frac{n-1}{2}-2}+3) +(2^{\frac{n-1}{2}}+\tfrac{(n-2)(n-1)}{2}-3)+2 \leq 2^{\frac{n+1}{2}}+1
\end{equation}
for $n \geq 21$. (We can do better but will refer later to \eqref{eqn:q2odd1stUB2}).

Next, suppose $H$ has exactly one edge. Without loss of generality take $\bm{yx}$ to be the edge. We split this case further. First suppose there exists an edge between $\bm{z}$ and $R_{\bm{x}} \cup R_{\bm{y}}$. Say, an edge between $\bm{z}$ and $\bm{v}\in R_{\bm{x}}$. We consider the decomposition \eqref{eqn:q2odddecomp} Then noting that $\bm{x}, \bm{z}\in T_{\bm{v}}$ and that Lemma~\ref{lem: induction}~\eqref{item:lem:induction:i} allows one to apply Lemma~\ref{lem:weak}, one recovers the same bound on $|S|$ as \eqref{eqn:q2odd1stUB}. Next, suppose there is no edge between $\bm{z}$ and $R_{\bm{x}}\cup R_{\bm{y}}$. In particular, $R_{\bm{x}} \cup \{\bm{y},\bm{z}\}$ is an orthogonal set. It follows that $V_{\bm{x}}$ does not have the maximum dimension that is possible for orthogonal subspaces. Thus, using the decomposition \eqref{eqn:q2odddecomp2}
and using Lemma~\ref{lem: induction}~\eqref{item:lem:induction:i}, Lemma~\ref{claim:Struct}, Lemma~\ref{lem:Rv} and Lemma~\ref{lem:weak}, we may obtain the same bound on $|S|$ as \eqref{eqn:q2odd1stUB2}.

Finally, suppose $H$ has two edges. Without loss of generality let $\bm{yxz}$ be the path of length 2. Here, we proceed to show that we may assume $|D|=3$, which, as pointed out earlier, gives the required result.

Suppose there exists a fourth vector $\bm{w}\in D$. If $\bm{w}$ forms an edge with $\bm{x}$, then there is no edge between any two of $\{\bm{y}, \bm{z}, \bm{w}\}$ and we are done by the arguments of the first case. If, on the other hand, $\bm{w}$ does not form an edge with $\bm{x}$, by Lemma~\ref{lem: induction}~\eqref{item:lem:induction:i}, dot product is non-degenerate on $\{\bm{x}, \bm{w}\}^{\perp}$, so we use
\[
S = S_{\bm{x}}\cup S_{\bm{w}} \cup S_{\bm{xw}} \cup\{\bm{x}, \bm{w}\}.
\]
Then, noting $|T_{\bm{x}}|\geq 2$ and arguing as before, we obtain
\begin{equation*}
|S| \leq (2^{\frac{n-1}{2}-2}+3) + (2^{\frac{n-1}{2}-1}+1) +(2^{\frac{n-1}{2}}+\tfrac{(n-2)(n-1)}{2}-3) + 2 \leq 2^{\frac{n+1}{2}}+1,
\end{equation*}
which is the same as \eqref{eqn:q2odd1stUB2}.
\end{proof}

The proof of Theorem~\ref{thm:q2even} is completed by combining Propositions~\ref{prop:Hy}, \ref{prop:q2evenI} and \ref{prop:q2evenII}.

\begin{remark}
\label{rem:2}
Theorem~\ref{thm:q2even} is false for small $n$. For $n=2$, $\S_{3,2}(2,2, \cdot) = 3$ as we see by taking $S = \F_2^2 \setminus \{\bm{0}\}$. For $n=4$ the example below shows $\S_{3,2}(2,4, \cdot) \geq 7$:
\[
S = \{ (1,1,1,0), (1,0,0,0),(1,0,1,1),(0,0,0,1),(0,1,1,1),(0,1,1,0), (1,1,0,1)\}.
\]
The graph of $S$ is indeed triangle-free: using the implicit order on the vertices, it is the union of the 6-cycle $234567$ with the edges $12$ and $47$.
\end{remark}

\section{Proof of Theorem~\ref{thm:k2OS}}
The following is essentially the same as \cite[Equation~2.4]{HIKR} and \cite[Lemma~5]{Shp}. Also see~\cite{AlKr} or apply the point-hyperplane incidence bound in~\cite{Vinh11}. 
\begin{lemma}
\label{leqm:OXYUB}
For $X, Y \subset \F_q^n$, define
$$
O(X, Y) = |\{(\bm{x}, \bm{y})\in X\times Y: \mathcal{B}(\bm{x}, \bm{y}) = 0 \}|.
$$
Then
$$
\bigg|O(X, Y) - \frac{|X||Y|}{q}\bigg| \leq \sqrt{|X||Y|q^n}.
$$
\end{lemma}

The following result is due to Tur\'an~\cite{Turan}.
\begin{lemma}
\label{lem:Turan}
Any graph of $n$ vertices, which is $K_{r+1}$-free contains at most $(1-1/r)(n^2/2)$ edges.
\end{lemma}
\begin{proof}[Proof of Theorem~\ref{thm:k2OS}]
Let $G = G(S, E_1)$ be the simple graph, where $(\bm{s_1}, \bm{s_2}) \in S^2$ forms an edge in $E_1$ if $\bm{s_1} \neq \bm{s_2}$ and $\mathcal{B}(\bm{s_1}, \bm{s_2})\not= \bm{0}$. Then using the fact that $S$ is $(k, 2)$-orthogonal, we know that $G$ is $K_{k}$-free and thus by Lemma~\ref{lem:Turan}, 
$$
|E_1|\leq \frac{k-2}{k-1}\frac{|S|^2}{2}.
$$
Denoting $G^{'} = G(S, E_2)$ as the complement of $G$, we deduce that
$$
|E_2| \geq \frac{|S|(|S|-1)}{2} - |E_1| \geq \frac{|S|^2}{2(k-1)} -\frac{|S|}{2}.
$$
Now, clearly $O(S,S) \geq 2|E_2|$. Hence, applying Lemma~\ref{leqm:OXYUB}, we have
\begin{equation*}
\frac{|S|^2}{k-1} - \frac{|S|^2}{q}-|S| \leq |S|q^{n/2},
\end{equation*}
which gives
\[
|S| \leq \bigg(\frac{q(k-1)}{q-k+1}\bigg)(q^{n/2}+1). \qedhere
\]
\end{proof}

\section*{Acknowledgement}
The authors are grateful to anonymous referees for their suggestions, which helped improve the presentation of the paper.

\end{document}